\documentclass[a4paper,11pt,reqno]{amsart}
\usepackage{amsmath,amsthm,amssymb}
\usepackage{enumerate}
\usepackage[utf8]{inputenc}
\usepackage[english]{babel}
\usepackage{tikz-cd}
\usepackage{dynkin-diagrams}

\makeatletter
\@namedef{subjclassname@2020}{\textup{2020} Mathematics Subject Classification}
\makeatother

\newtheorem{thm}{Theorem}[section]
\newtheorem*{thm*}{Theorem}
\newtheorem{lem}[thm]{Lemma}
\newtheorem{cor}[thm]{Corollary}
\newtheorem{prop}[thm]{Proposition}
\theoremstyle{definition}
\newtheorem{defin}[thm]{Definition}
\newtheorem*{exp*}{Example}

\newtheorem{rem}[thm]{Remark}

\newtheorem*{rem*}{Remark}

\numberwithin{equation}{section}

\def\P{{\mathbb P}}

\def\D{{\mathcal D}}

\def\C{{\mathcal C}}
\def\F{{\mathcal F}}
\let\Cal\mathcal
\let\Bbb\mathbb

\newcommand{\si}{\sigma}
\newcommand{\im}{\operatorname{im}}
\newcommand{\ps}{\psi}
\newcommand{\x}{\times}

\newcommand{\ga}{\gamma}

\newcommand{\cO}{{\mathcal O}}
\newcommand{\cC}{{\mathcal C}}
\newcommand{\cE}{{\mathcal E}}
\newcommand{\cL}{{\mathcal L}}
\newcommand{\cV}{{\mathcal V}}

\newcommand{\cH}{{\mathcal H}}
\newcommand{\cP}{{\mathcal P}}
\renewcommand{\o}{{\circ}}
\def\g{{\mathfrak g}}
\def\h{{\mathfrak h}}

\def\p{{\mathfrak p}}
\def\q{{\mathfrak q}}

\def\b{{\mathfrak b}}

\def\s{{\mathfrak s}}
\def\l{{\mathfrak l}}

\DeclareSymbolFont{script}{U}{eus}{m}{n}
\DeclareMathSymbol{\Wedge}{0}{script}{"5E}
\newcommand{\wh}{\widehat}

\newcommand{\Id}{\operatorname{Id}}
\newcommand{\id}{\operatorname{Id}}

\begin{document}

\title{Geometry of conic connections} 
\author{Andreas \v Cap}
\author{Katharina Neusser}
\address{
A.\v C.: Faculty of Mathematics\\
  University of Vienna\\
  Oskar-Morgenstern-Platz 1\\
  1090 Vienna\\
  Austria\\
K.N.: Department of Mathematics and Statistics\\
  Masaryk University\\
  Kotl\'a\v rsk\'a 267/2 \\
  611 37 Brno\\
  Czech Republic}

\email{andreas.cap@univie.ac.at}
\email{neusser@math.muni.cz}

\begin{abstract} 

  A cone structure on a complex manifold $M$ is a closed submanifold
  $\C\subset \P TM$ of the projectivized tangent bundle of $M$ that is submersive
  over $M$. So this defines a set $\C_x$ of distinguished directions in each point
  $x\in M$. A conic connection on $\C$ is then a family of unparametrized curves on
  $M$ that comprises exactly one curve through each point $x\in M$ in each direction
  in $\C_x\subset \mathbb PT_xM$. This can be encoded by a line subbundle
  $\F\subset T\C$. The subclass of characteristic conic connections is defined by the
  vanishing of a simple invariant, called characteristic torsion. For those, one has
  a much more subtle and slightly mysterious invariant called the cubic torsion. The
  first aim of this article is to provide a new approach to the cubic torsion, which
  also leads to a geometric condition characterizing its vanishing.
  
  We then specialize to the case of isotrivial cone structures, for which the fibers
  of $\C$ are assumed to be of some fixed type. Such a structure induces a
  first-order $G$-structure on $M$ whose structure group is the projective
  automorphism group of the model fiber. Moreover, any connection $\gamma$ on the
  associated $G$-structure induces a conic connection $\F^\gamma$ on $\C$. Assuming
  that the model fiber is homogeneous, we study the relation between the torsion and
  curvature of a connection $\gamma$ and the characteristic and cubic torsion of
  $\F^\gamma$. As an application we discuss cone structures of subadjoint type,
  showing in particular that there are such structures admitting conic connections
  with vanishing characteristic and cubic torsion that are not locally flat.
\end{abstract}
\maketitle

\section{Introduction}
A cone structure on complex manifold $M$ can be viewed as specifying a cone of
distinguished directions in each point of a complex manifold $M$. Formally, it is
defined as a closed submanifold $\cC$ of the projectivized tangent bundle
$\mathbb PTM$ of $M$ such that the natural projection $\cC\rightarrow M$ is a
surjective submersion. Such structures arise naturally in differential and algebraic
geometry, as well as mathematical physics. An important classical example is a
holomorphic conformal structure, which can most naturally be viewed as the
specification of a null cone $\cC_x\subset \mathbb PT_xM$ in each tangent space that
depends holomorphically on $x$, see \cite{LeBrun}. More precisely, associating to a
holomorphic conformal structure its bundle of null cones induces an equivalence of
categories between holomorphic conformal structures on $n$-dimensional manifolds and
cone structures whose fibers are projectively isomorphic to the standard quadric
hypersurface $Q^{n-2}\subset \mathbb C\mathbb P^{n-1}$. In the real smooth category,
Lorentzian and, more generally, pseudo-Riemannian conformal structures, can similarly
be described as cone structures. Weakening the conditions on the cones leads to
so-called causal structures, see \cite{Omid} and references therein, which have been
also studied in physics.  Even more generally, on has the large variety of cone
structures that underlie certain parabolic geometries \cite{H-N}, which
are geometries infinitesimally modeled on generalized flag varieties by means of
Cartan connections, see \cite{csbook}.

In algebraic geometry, certain cone structures, known as VMRT-structures, play an
important role in the study of uniruled projective manifolds. The latter are compact
complex manifolds that can be embedded into some projective space and that are
covered by rational curves. Roughly speaking, given an uniruled projective manifold
$X$, one can select a distinguished family of rational curves, called minimal
rational curves. Considering the tangent directions of these minimal rational curves
through a generic point $x\in X $ gives rise to a variety
$\C_x\subset \mathbb PT_xX$, called the variety of minimal rational tangents (VMRT)
at $x$.  If $\C_x\subset \mathbb PT_xX$ is a closed embedded submanifold, which holds
in many cases, then the union of the $\C_x$ as $x$ varies over general points gives
rise to a cone structure, called VMRT-structure, on an open dense subset of $X$. The
theory of VMRT-structures has been developed by Hwang and Mok and has been
successfully applied to solve many problems in algebraic geometry, see the surveys
\cite{HwangICM, Hwang-Mok99, Hwang26}.

For VMRT-structures, the cones arise as the tangent directions of minimal rational
curves. Similarly, the null cones of a holomorphic conformal structures can be viewed
as the tangent directions of a distinguished family of curves, namely its null
geodesics. For a general cone structure, the choice of a family of (unparametrized)
curves which similarly define the directions in the cone can be captured in the
notion of a conic connection. More precisely, given $\cC\subset\mathbb P TM$, a conic
connection on $\cC$ is defined as a certain line subbundle $\F\subset T\cC$ which has
trivial intersection with the vertical bundle of $\cC\to M$, see Definition
\ref{def_conic_conn}. The projections to $M$ of the integral submanifolds of the
distribution $\F$ then form a family of curves on $M$ with exactly one curve through
any point $x\in M$ in any direction of $\C_x$. This generalizes the classical notion
of a path geometry on a manifold $M$, which is simply a conic connection on the
entire projectivized tangent bundle $\cC=\mathbb P TM$ of $M$. As in the examples
mentioned above, in many cases of interest, cone structures of a fixed class come
equipped with a natural conic connection.  Despite the omnipresence of cone
structures and conic connections in mathematics and physics, a general theory just
started to emerge and the present article can be considered as a contribution to
these developments.
 
Our article is motivated by \cite{H-N}, where Hwang and the second author defined
characteristic conic connections by the vanishing of a rather simple local invariant
called the characteristic torsion. For characteristic conic connections, they
introduced a much more subtle invariant called the cubic torsion. For the
tautological conic connections on VMRT-structures, vanishing of the characteristic
torsion follows from the results of \cite{Hwang-Mok04} and in Proposition 3.5 of
\cite{H-N} it was shown that also the cubic torsion vanishes in this
case. Furthermore, Theorem 4.41 of \cite{H-N} showed that for many types of cone
structures, arising from parabolic geometries, existence of a conic connection with
vanishing characteristic and cubic torsion already determines the local isomorphism
type of the cone structure. These results lead then to an alternative proof,
respectively a local geometric version, of the recognition theorems by Mok \cite{Mok}
and Hong--Hwang \cite{H-H}, which show that certain generalized flag varieties
can be recognized from their VMRT's at general points.

Despite the important role these local invariants of conic connections have played in
\cite{H-N}, the cubic torsion remained a bit of a mysterious object. The original
definition arose from a formula for the so-called harmonic torsion of a path geometry
associated to a choice of a local section of $\F$, which was derived using technical
machinery for parabolic geometries. It then turned out that this formula can be
easily generalized to make sense for any characteristic conic connection on any cone
structure and this provided the definition of the cubic torsion in \cite{H-N}.

One aim of our article is to shed light on the cubic torsion and its
significance. This is done in Section 2, whose main result is Theorem
\ref{geom_interpr_cubic_tor}, which gives a geometric interpretation of the cubic
torsion. This interpretation generalizes the situation of path geometries, where it
is known that the harmonic torsion mentioned above is the obstruction for the space
of paths to inherit a Grassmannian structure (which, in the holomorphic category,
then turns out to be necessarily locally flat). The results and methods leading to
this theorem, in particular Proposition \ref{prop_Bott}, also provide an alternative
way to obtain or define the cubic torsion, which we believe clarifies its nature and
explains the explicit formula mentioned above.

From Section 3 on, we specialize to the subclass of isotrivial cone structures. This
means that for all $x\in M$, the cones $\C_x\subset \mathbb PT_xM$ are projectively
isomorphic to a fixed closed submanifold $Z\subset \mathbb P W$ in the
projectivization of a complex vector space $W$ of dimension $\dim M$.  Let us denote
by $\textrm{Aut}(\widehat Z)\subset\textrm{\textrm{GL}}(W)$ the group projective
automorphisms of $Z\subset \mathbb PW$.  To any $Z$-isotrivial cone structure on a
manifold $M$ one can associate a first-order $G$-structure with structure group
$\textrm{Aut}(\widehat Z)$ on $M$, i.e.\ a reduction $\cP$ of the frame bundle of $M$
to this structure group. Moreover, any (principal) connection $\gamma$ on $\cP$
induces a conic connection $\F^{\gamma}$ on $\cC$ given by the geodesics in
directions of $\cC$.

The study of conic connections on isotrivial cone structures induced by principal
connections on the associated $G$-structures was initiated by Hwang and Li in
\cite{Hwang-Li-2}, see also the recent survey \cite{Hwang26}. In particular, Theorem
3.8 of \cite{Hwang-Li-2} shows that the vanishing of the characteristic torsion of a
conic connection $\F^\gamma$ is equivalent to the torsion of the inducing principal
connection $\gamma$ having values in a certain subspace, which can be described in
terms of the projective geometry of $Z\subset \mathbb PW$. In addition, Theorem 3.6
of \cite{Hwang-Li-2} shows that for many types of isotrivial cone structures any
conic connection is locally induced from a principal connection on the associated
$G$-structure.

The basic invariants of a principal connection on a $G$-structure are of course its
curvature and torsion. Thus one might hope to obtain information on the cubic torsion
of a characteristic conic connection $\F^\gamma$ induced by a principal connection
$\gamma$ from the torsion and curvature of $\gamma$. Now for an isotrivial cone
structure $\cC\subset \mathbb PTM$ with corresponding $G$-structure $\cP\to M$, the
bundle $\cC\to M$ is an associated fiber bundle to $\cP$, but the tangent
space $T\cC$ is not easily accessible in this picture. Hence, in general, it is not
clear to what extent invariants of $\F^\gamma$ can be expressed in terms of
invariants of $\gamma$. The situation is much better, if one in addition assumes that
the modeling cone $Z\subset\mathbb PW$ is homogenous, i.e.\
$\textrm{Aut}(\widehat Z)$ acts transitively on $Z$. In this case, we derive in
Corollary \ref{cor_cubic_tor_to_curv} a formula for the cubic torsion of $\F^\gamma$
in terms of the torsion and curvature of $\gamma$. This leads to a characterization
for vanishing of the cubic torsion in the spirit of the results on the
characteristic torsion in \cite{Hwang-Li-2}. More specifically, we show in Theorem
\ref{thm_cubic_torsion_zero} that, under a certain condition on the torsion of
$\gamma$, the vanishing of the cubic torsion of $\F^\gamma$ is equivalent to the
curvature of $\gamma$ having values in a specific subspace, which can be described in
terms of the projective geometry of $Z\subset \mathbb PW$.

Section 4 applies our results to isotrivial cone structures, where the
homogeneous model cone $Z$ is a subadjoint variety as in Definition \ref{def_subadjoint}. For the real smooth version of
this case, the associated $G$-structures have been studied under the name parabolic
almost conformally symplectic structures (or PACS structures) by the first author and
Sala\v c in \cite{Cap_Salac}. They showed in particular that any PACS-structure
admits a canonical connection whose torsion satisfies a normalization
condition. For cone structure of subadjoint type $\neq B_3$ we show in Theorem
\ref{thm_conic_conn_subadjoint} that the existence of a conic connection with
vanishing characteristic torsion is equivalent to the existence of a torsion-free
connection on the associated PACS structure. We also prove that the cubic torsion of
any characteristic conic connection on these cone structures necessarily vanishes.

The canonical connections of PACS structures of the relevant types with vanishing
torsion turn out to be special symplectic connections in the sense of Cahen and
Schwachh\"ofer \cite{Cahen-Schwachhoefer}, see \cite{Cap_Salac}. Hence they realize
the special symplectic holonomies in the classification of irreducible holonomies of
torsion-free affine connections by Merkulov and Schwachh\"ofer \cite{MS}. Existence
of these connections, together with our aforementioned results, implies that there
are cone structures of subadjoint type admitting characteristic conic connections
with vanishing cubic torsion that are not locally flat. This is remarkable, since for
VMRT-structure of subadjoint type $\neq G_2$ it has been shown in \cite{Hwang-Li-1}
that they are locally flat. In contrast to the isotrivial cone structures considered
in \cite{H-N}, the condition of being a VMRT-structure for cone structure of
subadjoint type is thus more restrictive than admitting a conic connection with
vanishing characteristic and cubic torsion.

\smallskip \textbf{Acknowledgements} This article has been motivated by the joint
work \cite{H-N} of Jun-Muk Hwang and the second author. We would like to thank
Jun-Muk very much for many helpful discussions and comments, and the second author
also in particular for introducing her to the theory of VMRT-structures and many
inspiring conversations about them and related structures in recent years.  The
authors would also like to thank the Isaac Newton Institute for Mathematical
Sciences, Cambridge, for support and hospitality during the programme Twistor theory,
where the work on this paper has been initiated. This work was supported by EPSRC
grant EP/Z000580/1. The second author also acknowledges support by the grants
GA22-00091S from the Czech Science Foundation (GA\v CR) and MUNI/R/1435/2024 from the
grant agency of Masaryk University (GAMU). This article is based upon work from COST
Action CaLISTA CA21109 supported by COST (European Cooperation in Science and
Technology) https://www.cost.eu.

\section{Cone structures}\label{2}
We start by recalling some basics on cone structures and conic connections. 
We will work throughout this article in the holomorphic category, i.e.\,manifolds
will be complex and maps between them holomorphic. Moreover, for a holomorphic vector
bundle $\mathcal E$ over a complex manifold $M$ we denote by $\mathcal O(\mathcal E)$
the sheaf of local holomorphic sections.

\subsection{Cone structures and their tangential filtrations}\label{2.1}
For a complex vector space $W$, we denote by $\P W$ its projectivization and by
$\tau:W\setminus \{0\}\rightarrow \P W$ the natural projection.  For a subset
$Z\subset \P W$ we write
$$
\widehat Z:=\tau^{-1}(Z)\cup \{0\}\subset W,
$$
for the \emph{affine cone }over $Z$. In particular, if $z\in \P W$ is a point, $\hat z$
is the $1$-dimensional subspace in $W$ corresponding to $z$.
Recall that for a complex submanifold  $Z\subset \P W$ its \emph{affine tangent
  space} at a point $z_0\in Z$ is given by
$$
\wh{T}_{z_0}Z:= T_w\wh{Z}\subset W,
$$
where $w\in \hat z_0$ is any non-zero vector in $\hat z_0$. Note that this is
well-defined as the tangent spaces of $\wh Z$ along non-zero points in $\hat z_0$
coincide. The usual (intrinsic) tangent space at
$z_0$ is then related to the affine tangent space by the canonical isomorphism
$
T_{z_0}Z\cong \hat z_0^*\otimes \wh T_{z_0} Z/\hat z_0.
$

\begin{defin}\label{def_cone_str} 
Suppose $M$ is a complex manifold and write $\P TM$ for the projectivized (complex) tangent bundle of
$M$, whose fiber at $x\in M$ is the complex projective space $\P T_xM$.

\begin{enumerate}
\item A cone structure on $M$ is a closed submanifold $\mathcal C\subset \P TM$ such
  that the natural projection $p:\C\rightarrow M$ is a surjective submersion. So each
  fiber $p^{-1}(x)=\C_x$ is a closed submanifold of the projective space $\P T_xM$.
\item Given a cone structure $\mathcal C\subset \P TM$ on $M$, we denote by
  $\cV\cC\subset T\C$ the vertical bundle of the projection $p:\C\rightarrow M$.
\item Suppose that $W$ is a complex vector space of the same dimension as $M$ and
  that $Z\subset\P W$ is a closed submanifold. Then a cone structure $\C\subset\P TM$
  is called \textit{$Z$-isotrivial}, if $\C_x\subset \mathbb P T_xM$ is projectively
  isomorphic to $Z\subset \mathbb P W$ for all $x\in M$, that is, there is a linear isomorphism $T_xM\rightarrow W$ mapping $\widehat \C_x$
to $\widehat Z$ for all $x\in M$. 
\end{enumerate}
\end{defin}

Any cone structure $\C\subset \mathbb PTM$ admits a natural filtration of its tangent bundle by vector subbundles
\begin{equation}\label{natural_distributions_C}
T^{-1}\C\subset T^{-2}\C \subset T\C,
\end{equation}
where the fibers over a point $y\in\C$ are $T_y^{-1}\C:=(T_yp)^{-1}(\hat y)$ and $T_y^{-2}\C:=(T_yp)^{-1}(\wh T_y \C_{p(y)})$ respectively.
By construction, $\cV\cC$ is contained in $T^{-1}\C$ and of corank $1$ in $T^{-1}\C$. Hence, if $k$ is the dimension of the fibers of $p:\C\rightarrow M$, then the rank of $\cV\cC$ is $k$, the rank of $T^{-1}\C$ is $k+1$ 
and the rank of $T^{-2}\C$ is $2k+1$. Moreover, we have:

\begin{prop}\label{prop_filtration_C1} 
Suppose $\C\subset \mathbb PTM$ is a cone structure on a complex manifold. Then the subbundle $T^{-2}\C\subset T\C$ satisfies
$$\cO(T^{-2}\C)=[\cO(T^{-1}\C),\cO(T^{-1}\C)]+\cO(T^{-1}\C)=[\cO(\cV\cC),\cO(T^{-1}\C)]+\cO(T^{-1}\C).$$ 
\end{prop}
\begin{proof} The first identity was verified in Proposition 1 of \cite{Hwang-Mok04} and the second is evident, since $\cV\cC$ is of co-rank $1$ in $T^{-1}\C$.
\end{proof}

Proposition \ref{prop_filtration_C1} says in particular that $\C$ equipped with the filtration \eqref{natural_distributions_C} is a filtered manifold, that is, the filtration is compatible with the 
Lie bracket of vector fields in the sense that $[T^{-i}\C, T^{-j}\C]\subset T^{-i-j}\C$ for all $ i,j>0$, with the convention that $T^{-i}\C=T\C$ for $i\geq 3$. This implies that the Lie bracket of vector fields induces a vector bundle map
$\mathcal L: \textrm{gr}(T\C)\otimes\textrm{gr}(T\C)\rightarrow \textrm{gr}(T\C)$, called the \emph{Levi bracket},
on the associated graded bundle 
\begin{equation*}
\textrm{gr}(T\C):=\textrm{gr}_{-1}(T\C)\oplus\textrm{gr}_{-2}(T\C)\oplus \textrm{gr}_{-3}(T\C),
\end{equation*}
of the filtered bundle \eqref{natural_distributions_C},
where $\textrm{gr}_{-i}(T\C):=T^{-i}\C/T^{-i+1}\C$  with the convention that $T^{0}\C:=\{0\}$. Denoting by $q_{-i}: T^{-i}\C\rightarrow \textrm{gr}_{-i}(T\C)$ the natural projection,
 \begin{equation}\label{e.Levi}
\mathcal L: \textrm{gr}_{-i}(T\C)\otimes\textrm{gr}_{-j}(T\C)\rightarrow\textrm{gr}_{-i-j}(T\C),
\end{equation}
is defined by $\mathcal L(q_{-i}(\xi), q_{-j}(\eta))=q_{-i-j}([\xi,\eta])$ for local
sections $\xi$ and $\eta$ of $T^{-i}\C$ and $T^{-j}\C$ respectively. 

\begin{rem}\label{comparison_filt_rem}
  In general there exists two natural ways to continue the tangential filtration
  \eqref{natural_distributions_C}, but, as these extensions will not really play a
  role in this paper, we discuss them only briefly. Recall that for a submanifold
  $Z\subset \mathbb PW$, any choice of point $z_0\in Z$ induces a filtration
  $$W^{-1}\subset W^{-2}\subset...\subset W^{-r}\subset W,$$
  by linear subspaces, where $W^{-1}=\hat z_0$ and $W^{-2}=\wh T_{z_0}Z$. This is
  called the osculating filtration, which is induced by the projective fundamental
  forms of $Z\subset \mathbb P W$, see Section 2.1 of \cite{LM} for a
  definition. Therefore, the tangent space $T_y\C$ at $y\in\C$ of a cone structure
  $\C\subset \P TM$ admits a natural filtration by vector spaces
 \begin{equation}\label{filtration_C1}
\D_y^{-1}\subset \D_y^{-2}\subset...\subset \D _y^{-r}\subset T_y\mathcal C,
\end{equation}
given by the preimages under $T_yp$ of the osculating filtration of the subvariety
$\C_{p(y)}\subset \P T_{p(y)}M$ on $T_{p(y)}M$.  By definition $\D_y^{-1}=T_y^{-1}\C$
and $\D_y^{-2}=T_y^{-2}\C$, and if $\C_{p(y)}\subset \P T_{p(y)}M$ is
linearly-nondegenerate $\D _y^{-r}=T_y\mathcal C$.  At least on a dense open set of
$\C$, the filtration \eqref{filtration_C1} also gives rise to a filtration of $T\C$
by vector subbundles, extending the filtration \eqref{natural_distributions_C}.

There is a second natural tangential filtration on $\C$ given by considering the
weak-derived flag of the distribution $T^{-1}\C$. By its definition and Proposition
\ref{prop_filtration_C1}, the next subbundlle in that filtration is $T^{-2}\C$ and at
least, on an open dense subset of $\C$, it gives rise to a filtration of $T\C$ by
vector subbundles of the form
\begin{equation}\label{filtration_C2}
T^{-1}\C\subset T^{-2}\C\subset....\subset T^{-r'}\C\subset T\C.
\end{equation}
The filtration \eqref{filtration_C1} and \eqref{filtration_C2} in general do not
coincide: for $i\geq 3$ it is not difficult to prove that in general one only has
$\D^{-i}\subseteq T^{-i}\C$. This nevertheless implies that, if
$\C_{x}\subset \P T_{x}M$ is linearly non-degenerate for a generic point $x\in M$,
then $\D^{-i}=T^{-i}\C=T\C$ for sufficiently large $i$. So $T^{-1}\C$ is
bracket-generating in this case.  The question when the tangential filtrations
\eqref{filtration_C1} and \eqref{filtration_C2} on $\C$ coincide turns out to be
closely related to the existence of certain types of conic connections, which we
discuss in Section \ref{sec_conic_conn}, see also \cite{H-N}.
\end{rem}

\subsection{Conic connections}\label{sec_conic_conn}
\begin{defin}\label{def_conic_conn} 
  A \emph{conic connection} on a cone structure $\C\subset\P TM$ over a manifold $M$
  is a line subbundle $\F\subset T^{-1}\C$ such that $T^{-1}\C=\cV\cC\oplus \F$.
\end{defin}

Given $\F$, the projections of the integral curves of $\F\subset T\C$ to $M$ give
rise to a family of unparametrized holomorphic curves on $M$, with exactly one curve
through each point $x\in M$ in each direction belonging to $\C_x$. Conversely, any
such family lifting to a holomorphic foliation of rank $1$ on $\C$ defines a conic
connection on $\C$. If $\C=\P TM$, a conic connection is also called a \emph{path
  geometry} in the literature.

Given a conic connection $\F$, the involutivity of $\F$ and $\cV\cC$, and Proposition \ref{prop_filtration_C1} imply that the component  
\begin{equation}\label{Levi_iso}
\mathcal L: \F\otimes\cV\cC\rightarrow \textrm{gr}_{-2}(T\C)
\end{equation}
of the Levi bracket \eqref{e.Levi} is a surjection. By dimensional
reason, it thus has to be an isomorphism of vector bundles.  A significant geometric
property of a conic connection is the following \cite{HwangICM, H-N, Hwang-Li-1}:

\begin{defin} A conic connection $\F$ on a cone structure $\C\subset\P TM$ is called
  \emph{characteristic}, if
\begin{equation}
[\cO(\F), \cO(T^{-2}\C)]\subset \cO(T^{-2}\C).
\end{equation} 
This is equivalent to vanishing of the component
$\F\otimes\textrm{gr}_{-2}(T\C)\rightarrow T\C/T^{-2}\C$ of the Levi bracket
\eqref{e.Levi}, whose negative is called the \emph{characteristic torsion} of $\F$.
 \end{defin} 

In many cases, if a characteristic conic connection exists, it is known to be
unique:
 \begin{prop}\label{uniqueness_of_charc_conn}
 Suppose $\C\subset\P TM$ is a cone structure.
 \begin{enumerate}
 \item If the map
   $\cL: \cV\cC\rightarrow \emph{Hom}(\emph{gr}_{-2}(T\C), T\C/T^{-2}\C)$, induced by
   the corresponding component of the Levi bracket \eqref{e.Levi}, is injective at
   general points of $\C$, then there exists at most one characteristic conic
   connection.
 \item If each component of $\C_x\subset \P T_xM$ is different from a projective
   subspace for any $x\in M$, then the assumption in (1) is satisfied. In particular,
   $\C$ admits at most one characteristic conic connection.
 \end{enumerate}
 \end{prop}
 \begin{proof}
 For a proof see Proposition 1.21 and Lemma 1.22 of \cite{H-N} and the references therein. 
 \end{proof}
 
 For a conic connection with vanishing characteristic torsion there is an important
 local invariant, called the \emph{cubic torsion}, which was introduced in Section
 1.4 of \cite{H-N}. This is a section $\chi_{\F}$ of the bundle
 $S^3\F^*\otimes\textrm{Hom}_0(\cV\cC, \textrm{gr}_{-2}(T\C))$, where the subscript
 $0$ denotes trace-freeness with respect to the isomorphism
 $\cL: \F\otimes\cV\cC\cong \textrm{gr}_{-2}(T\C)$. Since the definition of
 $\chi_{\F}$ is rather technical and we will obtain an alternative derivation below,
 we do not discuss this definition from \cite{H-N} here. Proposition 3.5 of
 \cite{H-N} provides many examples of cone structures equipped with conic connections
 whose characteristic and cubic torsion vanishes, which arise naturally in the study
 of uniruled projective varieties. It was also shown in Theorem 4.41 of \cite{H-N},
 that there are many examples of isotrivial cone structures, arising in the context
 of parabolic geometries, for which the existence of a conic connection with
 vanishing characteristic and cubic torsion already determines the local isomorphism
 type of the cone structure.
 
 \subsection{Structures associated to a local holomorphic section of a characteristic
   conic connection $\F$.}\label{2.3}  
 Suppose $\C\subset\P TM$ is a cone structure on a complex manifold $M$ equipped with a conic
 connection $\F\subset T^{-1}\C$ so that $T^{-1}\C=\F\oplus\cV\C$.  Then we have the
 following commutative diagram with exact rows and columns:
\begin{equation}\label{com_exact_diagram}
\begin{array}{ccccccccc}
   &     &      0        &    &    0          & &  & & \\
  &     &      \downarrow  &     &    \downarrow     & & & & \\
  0 & \rightarrow & \F  & = & \F & \rightarrow&  0  &  \\
 &     &      \downarrow   &     &    \downarrow &   & \downarrow      & & \\
 0 &   \rightarrow  &    \cV\cC\oplus\F  &   \to  &   T^{-2} \C  &   \xrightarrow{q_{-2}}    & \textrm{gr}_{-2} (T\C) & \to& 0\\
  &     &      \downarrow         &     &      \downarrow       &    & \Vert & & \\
0 &  \to   &      \cV\cC      &  \rightarrow   &   T^{-2} \C/\F          &  \xrightarrow{q_{-2}}     &\textrm{gr}_{-2} (T\C) &\to &0 \\
 &     &       \downarrow     &     &    \downarrow         & &\downarrow & & \\
  &     &      0        &     &     0         & & 0& & \\
\end{array}.
\end{equation}
Recall that any short exact sequence of holomorphic vector bundles locally admits
holomorphic splittings and there are two equivalent descriptions of a splitting. For
example, a splitting of the middle row of \eqref{com_exact_diagram} is given by a
bundle map $\sigma:\textrm{gr}_{-2}(T\C)\to T^{-2} \C$ such that
$q_{-2}\circ\sigma=\textrm{Id}$. Equivalently, it can be viewed as a projection
\begin{equation}\label{notation_split}
\pi=\pi^{\cV\cC}+\pi^{\F}: T^{-2}\C\rightarrow\cV\cC\oplus\F 
\end{equation}
that restricts to the identity on $\cV\cC \oplus\F$. Given such a projection $\pi$,
$q_{-2}$ restricts to an isomorphism 
$\ker(\pi)\cong \textrm{gr}_{-2} (T\C)$ and we get the following identifications
\begin{equation}\label{direct_sum}
T^{-2}\C\cong\ker(\pi)\oplus\cV\cC \oplus\F\cong \textrm{gr}_{-2}(T\C)\oplus\cV\cC \oplus\F.
\end{equation}
Evidently, we get also an induced splitting of the bottom row
\begin{equation*}
T^{-2}\C/\F\cong\ker(\pi)\oplus\cV\cC \cong \textrm{gr}_{-2}(T\C)\oplus\cV\cC .
\end{equation*}

For a characteristic conic connection $\F$, we next associate several data to a
choice of a local nowhere vanishing section $f$ of $\F$, i.e.\ a local
trivialization of $\F$. These are analogous to several components of Weyl
structures for path geometries, see \cite{Cap-Guo}, and Chapter 5 of \cite{csbook} for
the general concept for parabolic geometries. The resulting data include a local
splitting as in \eqref{notation_split}, as well as a partial linear connection on the vector bundle $\cV\C$. 
Since partial connections will in general play a crucial
role in this article, we recall  here the definition for the convenience of the reader.
\begin{defin}\label{def_partial_conn}
 Let $\mathcal E\to M$ be a (holomorphic) vector bundle over a complex manifold $M$ and 
 $\mathcal H\subset TM$ be a (holomorphic) distribution on $M$. Then a (holomorphic) \textit{partial
 $\mathcal H$-connection on $\mathcal E$} is a $\Bbb C$-bilinear map $\nabla: \cO(\mathcal
  H)\times \cO(\mathcal E)\to\cO(\mathcal E)$ which has the defining properties of a linear
  connection. Explicitly, this means that for local sections $\xi$ of $\mathcal
  H$, $s$ of $\mathcal E$ and a local holomorphic function $h$ one has
\begin{enumerate}
\item $\nabla_{h\xi}s=h\nabla_\xi s$,
\item $\nabla_\xi h s=h\nabla_\xi s+dh(\xi)s$.
\end{enumerate}
If the subbundle $\mathcal H$ is clear from the context, we will refer to such a $\nabla$ just as a
partial connection.
\end{defin}

\begin{prop}\label{prop_splitting_f} 
  Suppose $\C\subset\P TM$ is a cone structure on a manifold $M$ equipped with a
  characteristic conic connection $\F$. Then any choice of local nowhere vanishing
  section $f$ of $\F$ induces the following data on the domain of definition of $f$:
\begin{enumerate}
\item A vector bundle isomorphism $\phi_f=\mathcal L(f,_-):\cV\cC
\rightarrow\emph{gr}_{-2}(T\C)$.
\item A partial $\F$-connection $\nabla^f:\cO(\cV\cC )\rightarrow\cO(\F^*\otimes\cV\cC)$ on $\cV\cC $ characterized by 
\begin{equation}\label{def_nabla^f}
\mathcal L(f, \nabla_f^fv)=\frac{1}{2}q_{-2}([f,[f,v]]).
\end{equation}
\item A partial $\F$-connection
$\nabla^f:\cO(\emph{gr}_{-2}(T\C))\rightarrow\cO(\F^*\otimes(\emph{gr}_{-2}(T\C))$
on $\emph{gr}_{-2}(T\C)$ obtained from (2) by requiring compatibility with $\phi_f$.
\item A splitting $\pi_f=\pi_f^{\cV\cC} +\pi_f^\F:T^{-2}\C\rightarrow
T^{-1}\C=\cV\cC\oplus\F $ of the middle row of \eqref{com_exact_diagram} given by
\begin{equation}\label{splitting_f}
\pi_f(\xi)=\xi-[f,v_{q_{-2}(\xi)}]+\nabla_f^{f}v_{q_{-2}(\xi)},
\end{equation}
where $v_{q_{-2}(\xi)}=\phi_f^{-1}(q_{-2}(\xi))$ for a local section  $\xi$ of $T^{-2}\C$. 
In particular, 
$\pi_f([f,v])=\pi_f^{\cV\cC} ([f,v])=\nabla^f_f v$ for a local section $v$ of $\cV\cC$.
\end{enumerate}
\end{prop}
\begin{proof}
(1) We have already observed in \eqref{Levi_iso} that the Levi bracket induces an
  isomorphism $\phi_f$ as required.

(2) Since $\F$ is characteristic, for $v\in\cO(\cV\cC)$ we get
  $[f,[f,v]]\in\cO(T^{-2}\C)$. So, by (1), there is a unique section $\nabla^f_f
  v\in\cO(\cV\cC)$ such that \eqref{def_nabla^f} holds. Any local section of $\F$
  is of the form $hf$ for some locally defined holomorphic function $h$ on $\C$ and
  we define
\begin{equation}\label{def_nabla^f_2}
\nabla^f_{hf}v:=h\nabla_f^fv. 
\end{equation}
To prove that we have defined a partial connection, it suffices to verify the Leibniz
rule in the second argument, which is a simple direct computation, see (1) of
Lemma 1.26 of \cite{H-N}, where $\nabla_f^fv$ is denoted by $w(f,v)$. 

(3) is obvious. 

(4) For a local section $\xi$ of $T^{-2}\C$ put
$v_{q_{-2}(\xi)}:=\phi_f^{-1}(q_{-2}(\xi))\in\cO(\cV\cC)$. Then $\pi_f(\xi)$ as
defined in \eqref{splitting_f} is a section of $T^{-1}\C$, since one obviously has $q_{-2}(\pi_f(\xi))=q_{-2}(\xi)-\phi_f(v_{q_{-2}(\xi)})=0$.  Observe also
that for $\xi=[f,v]$, where $v\in\mathcal O(\cV\C)$, we get $q_{-2}(\xi)=\Cal L(f,v)=\phi_f(v)$, which shows that
$\pi_f([f,v])=\nabla_f^{f}v$.  Let us now verify that $\pi_f$ is a vector bundle map:
suppose $h$ is a locally defined holomorphic function on $\C$, then
$v_{q_{-2}(h\xi)}=hv_{q_{-2}(\xi)}$ implies that
\begin{align*}
\pi_f(h\xi)&=h\xi-[f,hv_{q_{-2}(\xi)}]+\nabla^f_fhv_{q_{-2}(\xi)}\\
&=h\xi-h[f,v_{q_{-2}(\xi)}]-dh(f)v_{q_{-2}(\xi)}+h\nabla_f^fv_{q_{-2}(\xi)}+dh(f)v_{q_{-2}(\xi)}\\
&=h\pi_f(\xi),
\end{align*}
which shows that $\pi_f$ defines a vector bundle map. Moreover, obviously,
$\pi_f\vert_\F=\textrm{Id}_{\F}$ and $\pi_f\vert_{\cV\cC}=\textrm{Id}_{\cV\cC }$.
Hence, \eqref{splitting_f} defines a splitting as claimed.
\end{proof}

The following proposition establish how the objects in Proposition \ref{prop_splitting_f} 
change when one changes the local trivialization of $\F$.

\begin{prop}\label{prop_splitting_f_2}
In the setting of Proposition \ref{prop_splitting_f}, suppose $f$ and $\hat f$ are
local nowhere vanishing sections of $\F$, which implies $\hat f=hf$ for some local
nowhere vanishing holomorphic function $h$ on $\C$. Then one has:
\begin{equation}\label{change_phi_f}
\phi_{\hat{f}}=\phi_{hf}=h\phi_f\quad\textrm{ and }\quad\phi_{\hat{f}}^{-1}=\phi_{hf}^{-1}=h^{-1}\phi_f,
\end{equation}
\begin{equation}\label{change_nabla^f}
\nabla^{\hat f}_gv=\nabla^{hf}_gv=\nabla^f_gv+\frac{1}{2h}dh(g)v,
\end{equation}
\begin{equation}\label{change_pi_f}
\pi_{\hat f}^{\cV\cC} (\xi)=\pi_{f}^{\cV\cC} (\xi)+\frac{1}{2h}dh(f)\phi_f^{-1}(q_{-2}(\xi))=\pi_{f}^{\cV\cC}(\xi)+\frac{1}{2h}dh(f)v_{q_{-2}(\xi)},
\end{equation}
for local sections $g\in\cO(\F)$, $v\in\cO(\cV\cC)$, and $\xi\in\cO (T^{-2}\C)$.
\end{prop}
\begin{proof}
The identities \eqref{change_phi_f} are obvious and the identity
\eqref{change_nabla^f} was shown in Lemma 1.26 (3) of \cite{H-N}. By the identities
\eqref{change_phi_f} and \eqref{change_nabla^f}, and putting
$v=v_{q_{-2}(\xi)}=\phi_f^{-1}(q_{-2}(\xi))$, we have
\begin{align*}
\pi_{hf}(\xi)&=\xi-[hf, h^{-1}v]+\nabla_{hf}^{hf}h^{-1}v\\
&=\xi-h[f, h^{-1}v]+h^{-1}dh(v)f +h^{-1}\nabla_{hf}^{hf}v+hdh^{-1}(f)v\\
&=\xi-[f,v]+h^{-1}dh(v)f+h^{-1}(\nabla^f_{hf}v+\frac{1}{2}dh(f)v)\\
&=\pi_f(\xi)+\frac{1}{2h}dh(f)v+\frac{1}{h}dh(v)f,
\end{align*}
which implies the claim.
\end{proof}

Given a cone structure $\C\subset\P TM$ endowed with a characteristic conic
connection $\F$, Proposition \ref{prop_splitting_f} shows that for any local nowhere
vanishing section $f$ of $\F$, the projection $\pi_f$ induces a local isomorphism
\begin{equation}\label{isos_induced_by_f}
T^{-2}\C/\F\cong \cV\cC\oplus \bar{q}(\ker \pi_f^{\cV\cC })
\cong\cV\cC\oplus\textrm{gr}_{-2}(T\C),
\end{equation}
where $\bar{q}: T^{-2}\C\rightarrow T^{-2}\C/\F$ is the natural projection, and a
locally defined isomorphism $\phi_f:\cV\cC\to \textrm{gr}_{-2}(T\C)$ between the two
summands. This structure depends on $f$ and the explicit form of the dependence can
be read off the last part of Proposition \ref{prop_splitting_f_2}. Using this, we
will show that a weakening of this decomposition is independent of $f$ and we can
then study the question of whether this descends to local leaf spaces for $\F$.

\subsection{Split-quaternionic structures and Segre structures}\label{2.4new}
While most of these topics are, at least in the real smooth setting, available in
literature, they seem to be not well known. Since there are some subtleties
involved, we decided to add some detail. Most of this is a matter of linear algebra,
so we start the discussion in this setting.

Consider a complex vector space $E$ of even dimension $2n$ endowed with a linear
subspace $V\subset E$ of dimension $n$. What we obtained at the end of Section
\ref{2.3} corresponds to a decomposition $E=V\oplus\tilde V$ together with a
distinguished linear isomorphism $\tilde V\to V$. The choice of $\tilde V$ is
equivalent to a splitting of the short exact sequence $0\to V\to E\to E/V\to 0$ and
we will encode this via a linear map $\sigma:E/V\to E$ such that $q\circ\sigma=\id$
for the natural quotient projection $q:E\to E/V$. The isomorphism $\tilde V\to V$
will in turn be encoded as $\psi\circ q$ for a linear isomorphism $\psi:E/V\to
V$. For our purpose, a different encoding of these data will be useful.

\begin{lem}\label{lem:IJK}
  Let $E$ be a complex vector space of even dimension $2n$. Choosing an
  $n$-dimensional subspace $V\subset E$, a splitting of $q:E\to E/V$ and a linear
  isomorphism $\psi:E/V\to V$ is equivalent to choosing a pair $I,J:E\to 
   E$ of linear isomorphisms such that $I^2=J^2=\Id$ and $I\o J=-J\o I$.
\end{lem}
\begin{proof}
  Given $V$, a splitting $\sigma:E/V\to E$ and a linear isomorphism $\psi:E/V\to V$,
  we obtain an isomorphism $E\cong V\oplus V$. Explicitly, a vector $w\in E$ can be
  uniquely written as $v_1+\si(\psi^{-1}(v_2))$, where $v_2=\psi(q(w))$ and
  $v_1=w-\si(q(w))$ and we identify $w$ with the pair $(v_1,v_2)$. We then define
  linear maps $I,J:E\to E$ as $I(v_1,v_2)=(v_1,-v_2)$ and
  $J(v_1,v_2)=(v_2,v_1)$. This readily implies that $I^2=J^2=\id$ and that
  $K:=I\circ J=-J\circ I$, which is explicitly given by $K(v_1,v_2)=(v_2,-v_1)$,
  satisfies $K^2=-\id$.

For later use, we note that the maps $I$, $J$, and $K$ are explicitly given by
\begin{equation}\label{IJK}
\begin{gathered}
  I(w)=w-2\sigma(q(w))\\
  J(w)=\psi(q(w))+\sigma(\psi^{-1}(w-\sigma(q(w))))\\
  K(w)=\psi(q(w))-\sigma(\psi^{-1}(w-\sigma(q(w)))).
\end{gathered}
\end{equation}

Conversely, given $I, J:E\to E$, such that $I^2=J^2=\id$ and $I\circ J=-J\circ I$, it
follows that $I$ and $J$ are diagonalizable with eigenvalues $+1$ and $-1$ and that
$J$ maps the $+1$-eigenspace of $I$ isomorphically onto the $-1$-eigenspace of
$I$. We then define $V\subset E$ to be the $+1$-eigenspace of $I$, $\sigma(w+V)$ as
the component of $w$ in the $-1$-eigenspace of $I$ and $\psi:=J\o \sigma:E/V\to
V$. This obviously is inverse to the above construction.
\end{proof}

Taking the isomorphism $E\cong V\oplus V$ from the proof of this lemma and
interpreting the components as the ``columns of a matrix'', we can also view this as
an isomorphisms $E\cong \textrm{Hom}(\Bbb C^2,V)$ which sends $w\in E$ to the map
$(a,b)\mapsto av_1+bv_2=aw-a\si(q(w))+b\psi(q(w))$. In this picture, the linear
isomorphisms $I$, $J$, and $K:=I\o J$ correspond to precomposition by the matrices
$\left(\begin{smallmatrix} 1 & 0 \\ 0 & -1\end{smallmatrix}\right)$,
  $\left(\begin{smallmatrix} 0 & 1 \\ 1 & 0\end{smallmatrix}\right)$, and
    $\left(\begin{smallmatrix} 0 & 1 \\ -1 & 0\end{smallmatrix}\right)$,
      respectively. In particular, we readily conclude that $\id_E$, $I$, $J$, and
      $K$ span a $4$-dimensional unital subalgebra of $\textrm{Hom}(E,E)$ that is
      isomorphic to $M_{2\times 2}(\Bbb C)$.

      A natural weakening of this structure is replacing the choice of the maps $I$,
      $J$, and $K=I\o J$ by the choice of the three dimensional subspace of
      $\textrm{Hom}(E,E)$ they span, or equivalently of the subalgebra of
      $\textrm{Hom}(E,E)$ they span together with $\id_E$. It turns out that this
      also admits a nice description in the picture of the identification of $E$ with
      $\textrm{Hom}(\Bbb C^2,V)$. The space $\textrm{Hom}(\Bbb C^2,V)$ contains the
      cone $\widehat{Z}_0$ of maps of rank $<2$ and we call its pre-image
      $\widehat{Z}\subset E$, the \textit{Segre cone} in $E$ defined by the
      identification. We are now ready to formulate the results we will need in the
      setting of linear algebra.

\begin{lem}\label{lem:split-quat}
  Let $E$ be a complex vector space of even dimension $2n$. Then 

  \begin{enumerate}
  \item Suppose that we have pairs $(I,J)$ and $(\hat I,\hat J)$ as in Lemma
    \ref{lem:IJK} such that $I$, $J$ and $K=I\circ J$ and $\hat I$, $\hat J$, and
    $\hat K=\hat I\circ\hat J$ span the same linear subspace of
    $\emph{Hom}(E,E)$. Then the two pairs give rise to the same Segre cone in $E$.

  \item In particular, this happens if in the setting of Lemma \ref{lem:IJK}, we
    start from data $(V,\si,\psi)$ and $(\hat V,\hat\si,\hat\psi)$ that are related
    by $\hat V=V$, $\hat\psi=a\psi$ and $\hat\si=\si+b\psi$ for $a,b\in\Bbb C$ with
    $a\neq 0$.
  \end{enumerate}
\end{lem}
\begin{proof}
  (1) We show that the elements in the Segre cone $\widehat{Z}\subset\textrm{Hom}(E,E)$
  obtained from a pair $(I,J)$ can be characterized in terms of the subalgebra
  $\mathcal A\subset \textrm{Hom}(E,E)$ spanned by $I$, $J$, $K=I\o J$ and $\id_E$,
  which of course implies the claimed result. Given an element $w\in E$ corresponding
  to a linear map $f:\Bbb C^2\to V$ we can look at the images of $w$ under all maps
  in $\mathcal A$, which clearly form a linear subspace $\mathcal A(w)\subset
  E$. From the above description of the actions of elements of $\mathcal A$ in terms
  of pre-compositions we readily conclude that this corresponds to the subspace of
  $\textrm{Hom}(\Bbb C^2,V)$ of the form $f\o g$ for a linear map $g:\Bbb C^2\to\Bbb
  C^2$. But this is exactly the subspace of all linear maps $\Bbb C^2\to V$ whose
  image is contained in the image of $f$, so its dimension is twice the rank of
  $f$. Hence $w$ lies in the Segre cone $\widehat{Z}$ if and only if $\dim(\mathcal
  A(w))\leq 2$.
        
\smallskip
        
(2) We use the explicit formulae for $I$, $J$, and $K$ from \eqref{IJK}. First, we
get $\psi(q(w))=\frac12(J(w)+K(w))$. Now by assumption
$$
\hat I(w)=w-2\hat\si(q(w))=w-2\si(q(w))-2b\psi(q(w)),
$$
so $\hat I$ is a linear combination of $I$ and $\frac12(J+K)$. Our assumptions
also imply directly that $\frac12(\hat J+\hat K)=\frac{a}2(J+K)$, so to complete the
proof, it suffices to show that $\frac12(\hat J-\hat
K)=\hat\si\hat\psi^{-1}(w-\hat\si(q(w)))$ lies in the span of $I,J$ and $K$. For this
we first observe that $w-\hat\si(q(w))=w-\si(q(w))-b\ps(q(w))$ and applying
$\hat\psi^{-1}=a^{-1}\psi^{-1}$ to this gives
$a^{-1}\psi^{-1}(w-\si(q(w)))-a^{-1}bq(w)$. Finally, we have to apply
$\hat\si=\si+b\ps$ to this, which gives
  \begin{equation*}
 \quad\quad a^{-1}\si\psi^{-1}(w-\si(q(w)))+a^{-1}b(w-2\si(q(w)))+a^{-1}b^2\ps(q(w)).  
  \end{equation*}
  Now the first term is a multiple of $(J-K)(w)$, the second is a multiple of
  $I(w)$ and the last one a multiple of $(J+K)(w)$.
\end{proof}

Now we are ready to move to the picture of geometric structures on vector bundles and
we first collect the necessary definitions. Observe that the notion of an isotrivial
cone structure from Definition \ref{def_cone_str} (3) naturally extends to the
setting of holomorphic vector bundles over complex manifolds. For a holomorphic
vector bundle $\mathcal{E}\to M$, we can naturally form the projectivization
$\P {\mathcal E}\to M$. Given a complex vector space $E$ whose dimension equals the
rank of $\mathcal E$ and viewing $\mathcal E$ as modeled on $E$, we can view
$\P\mathcal E$ as a fiber bundle with typical fiber $\P E$. Given a closed
submanifold $Z\subset \P E$, there is thus the natural notion of a $Z$-isotrivial
cone structure on $\mathcal E$ as a subbundle $\C \subset\P\mathcal E$ which is
modeled on $Z\subset\P E$. Using this observation, we make the following
definitions.

\begin{defin}\label{def-spli-quat}
  Let $M$ be a complex manifold and let $\mathcal{E}\to M$ be a vector bundle of even
  rank $2n$.

  (1) A \textit{Segre structure} on $\mathcal{E}$ is a $Z$-isotrivial cone structure
  $\C\subset \P \mathcal{E}$ on $\mathcal{E}$ modeled on the cone
  $Z\subset\P \textrm{Hom}(\Bbb C^2,\Bbb C^n)$ of maps of rank $1$.

  (2) A (complex) \textit{split quaternionic structure} on $\mathcal{E}$ is a
  subbundle $Q\subset \textrm{Hom}(\mathcal{E},\mathcal{E})$ of rank $3$, which can
  be locally spanned by sections $I$, $J$ and $K=I\circ J$ such that $I^2=J^2=\id$
  and $I\circ J=-J\circ I$. Equivalently, one can add the identity and view the
  structure as a bundle of associative unital subalgebras in
  $\textrm{Hom}(\mathcal{E},\mathcal{E})$ modeled on the algebra
  $M_{2\times 2}(\Bbb C)$.
\end{defin}

\begin{rem}\label{rem-split-quat}
  The motivation for the terminology in (2) comes from the concept of (almost)
  quaternionic structures in the real smooth setting together with the fact that
  $M_{2\times 2}(\Bbb R)$ is a split-signature analog of the quaternions. Thus
  $M_{2\times 2}(\Bbb R)$ is also called the algebra of split quaternions. In this
  spirit, the relations $I^2=J^2=\id$, $IJ=-JI$ are called split quaternion
  relations. Over $\Bbb C$, there is no difference between quaternions and split
  quaternions, so one could also call the structure in (2) a complex quaternionic
  structure. Since our focus is on local generators which satisfy the split
  quaternionion relations, we prefer to use the name split-quaternionic structure. In
  the real smooth setting (where there is a difference to quaterioninc structures),
  there is also the terminology ``para-quaternionic structure'', but this seems less
  descriptive to us.
\end{rem}

\begin{lem}\label{lem-split-quat}
  Let $M$ be a complex manifold and $\mathcal{E}\to M$ a holomorphic vector bundle of
  even rank $2n$. Then a split quaternionic structure
  $Q\subset \emph{Hom}(\mathcal{E},\mathcal{E})$ on $\mathcal{E}$ canonically
  induces a Segre structure on $\mathcal{E}$.
\end{lem}
\begin{proof}
  This is a direct translation of the results from linear algebra. Starting from
  $Q\subset \textrm{Hom}(\mathcal E,\mathcal E)$, we choose local sections $I$ and
  $J$ as in the definition. Then the pointwise $+1$-eigenspaces of $I$ form a
  holomorphic subbundle $\mathcal V$ of $\mathcal E$ of rank $n$. The induced
  pointwise identifications $\mathcal E_x\cong\textrm{Hom}(\Bbb C^2,\mathcal V_x)$
  define a Segre cone in each fiber $\mathcal E_x$, so locally we obtain a Segre
  structure. By Lemma \ref{lem:split-quat} the Segre cones depend only on $Q$ and not
  on the choices of $I$ and $J$, so the local cone structures fit together to define
  a global Segre structure on $\mathcal E$ which is canonically associated to $Q$.
\end{proof}

\begin{rem}
  In the real smooth setting, it is well known that a Segre structure in dimension
  $2n$ is equivalent to a local identification with a tensor product of auxiliary
  bundles of rank $2$ and $n$. In \cite{Zadnik} it is shown that these structures are
  in turn equivalent to a split quaternionic structure. All these facts indeed extend
  to the holomorphic setting, but since we don't formally need them, we just give a
  brief outline of the argument. The key ingredient for this is the description of
  the automorphism group $\textrm{Aut}(\widehat{Z}_0)$ of the cone $\widehat{Z}_0$ of maps of rank
  $<2$ in $\textrm{Hom}(\Bbb C^2,\Bbb C^n)$. This group is the image of
  $\textrm{GL}(n,\Bbb C)\times \textrm{GL}(2,\Bbb C)$ under the homomorphism to
  $\textrm{GL}(\textrm{Hom}(\Bbb C^2,\Bbb C^n))$ coming from composition from both
  sides, see e.g.\ Appendix A of \cite{Mettler} for a proof.

  The main subtlety comes from the fact that, depending on the parity of $n$, this
  image is either isomorphic to
  $\textrm{S}(\textrm{GL}(n,\Bbb C)\times \textrm{GL}(2,\Bbb C))$ or to the quotient
  of this group by the two element subgroup
  $\{(\Bbb I_n,\Bbb I_2),(-\Bbb I_n,-\Bbb I_2)\}$. Now a Segre structure on
  $\mathcal E$ defines a reduction of structure group of the frame bundle of
  $\mathcal E$ to the group $\textrm{Aut}(\widehat Z_0)$, compare to Section
  \ref{sec_assoc_G-structure} below, which locally lifts to a two-fold covering
  without problems. The groups
  $\textrm{S}(\textrm{GL}(n,\Bbb C)\times \textrm{GL}(2,\Bbb C))$ naturally acts on
  $\Bbb C^{2*}$ and $\Bbb C^n$, and the tensor product of the corresponding
  associated bundles is isomorphic to $\mathcal E$.

  Via the local isomorphism to a tensor product, acting on the two dimensional factor
  in the tensor product gives, for each point $x$, a subalgebra in
  $\textrm{Hom}(\mathcal{E}_x,\mathcal E_x)$ isomorphic to $M_{2\times 2}(\Bbb C)$
  and we define $Q_x$ to be the subspace formed by trace-free-matrices. Locally,
  these fit together to a holomorphic subbundle in
  $\textrm{Hom}(\mathcal{E},\mathcal{E})$, which of course can be spanned by
  appropriate sections $I$, $J$ and $K$. Since the only non-uniqueness in the tensor
  product representation is a sign change, this globalizes to a split quaternionic
  structure on $\mathcal {E}$.
\end{rem}

\subsection{The split-quaternionic structure on $T^{-2}\C/\F$ induced by a
  characteristic conic connection $\F$.}
Let us return to the situation of a cone structure $\C\subset\P TM$ endowed with a
characteristic conic connection $\F$. Then we have the short exact sequence
$0\to \cV\C\to T^{-2}\C/\F \to \textrm{gr}_{-2}(T\C)\to 0$ from the diagram
\eqref{com_exact_diagram}. In Proposition \ref{prop_splitting_f}, we have associated
to a local nowhere vanishing section $f\in\cO(\F)$ in particular a splitting of this
sequence as well as an linear isomorphism $\phi_f: \cV\C\to
\textrm{gr}_{-2}(T\C)$. By part (1) of Lemma \ref{lem:IJK}, these data can be
equivalently encoded in terms of locally defined endomorphisms $I_f$ and $J_f$ of
$T^{-2}\C/\F$ which satisfy the split quaternion relations. Whereas the pair 
$(I_f,J_f)$ depends on $f$, we can now show that the induced rank $3$ subbundle in
$\textrm{End}(T^{-2}\C/\F)$ does not and hence that one obtains a canonical (globally
defined) split quaternionic structure on $T^{-2}\C/\F$.

\begin{prop}\label{para-q-str}
  The local data $(I_f,J_f)$ induce a globally defined split quaternionic structure
  $Q\subset \emph{End}(T^{-2}\C/\F)$ on the vector bundle $T^{-2}\C/\F$.
\end{prop}
\begin{proof}
 We need to show that the a priori only locally defined subbundle $Q_f\subset \textrm{End}(T^{-2}\C/\F)$ spanned by $I_f$, $J_f$ and $K_f=I_f\circ J_f $
 is independent of the choice of local nowhere vanishing section $f\in\cO(\F)$. We use the results on the dependence of the data on $f$ from Proposition
 \ref{prop_splitting_f_2} to show that the second statement in part (2) of Lemma
 \ref{lem:split-quat} applies in our situation. The isomorphism $\psi$ used in that
 lemma is simply the map $\phi_f^{-1}$ from Proposition \ref{prop_splitting_f}, which
 point-wise just changes by a non-zero multiple by equation \eqref{change_phi_f}. On
 the other hand, the splitting $\si$ is characterized by $\si(q_{-2}(\xi))=\bar
 q(\xi)-\pi_f^{\cV\cC}(\xi)$. But equation \eqref{change_pi_f} exactly says that in
 each point, the change of $-\pi_f^{\cV\cC}$ is given by adding a multiple of
 $\phi_f^{-1}(q_{-2}(\xi))$. Of course, the factors computed in the proof of part (2)
 of Lemma \ref{lem:split-quat} then depend holomorphically on the base point and the
 claim follows.
\end{proof}

\subsection{The Bott connection and the cubic torsion}\label{2.6}
Suppose $\C\subset\mathbb P TM$ is a cone structure with a characteristic conic
connection $\F$. Since one has $[\cO(\F), \cO(T^{-2}\C)]\subset \cO(T^{-2}\C)$, the
vector bundle $T^{-2}\C/\F$ descends to a distribution $\mathcal D\subset TN$ on any
local leaf space $N$ for $\F$. Thus it is natural to ask whether the split
quaternionic structure obtained in Proposition \ref {para-q-str} descends to a split
quaternionic structure on $\mathcal D$. Our aim is to show that this happens if and
only if the cubic torsion of $\F$ vanishes identically.

In foliation theory, it is well known that for an involutive subbundle of a tangent
bundle, one can use the Lie bracket with sections of this subbundle to define a
partial connection on the quotient of the tangent bundle by the subbundle. To adapt
this to our setting, let $\C\subset\P TM$ be a cone structure on a manifold $M$
equipped with a characteristic conic connection $\F\subset T^{-1}\C$. As before, let
us denote by $\bar q: T^{-2}\C\rightarrow T^{-2}\C/\F$ the natural projection.  Note
that for local sections $f\in\cO(\F)$ and $\xi\in\cO(T^{-2}\C)$, the section
$\bar q([f,\xi])\in\cO(T^{-2}\C/\F)$ depends only on $\bar q(\xi)$ and one
verifies directly:
\begin{prop}
For a characteristic conic connection $\F$ on a cone structure $\C\subset\P TM$
consider $\nabla^{\mathcal B} : \cO(\F)\times \cO(T^{-2}\C/\F)\rightarrow
\cO(T^{-2}\C/\F)$ characterized by
\begin{equation*} 
\nabla^{\mathcal B}_f \bar q(\xi):=\bar q([f,\xi])
\end{equation*}
for $f\in\cO(\F)$ and $\xi\in\cO(T^{-2}\C)$. Then $\nabla^{\mathcal B}$ is a
well defined partial $\F$-connection on $T^{-2}\C/\F$, which we refer to as the Bott
connection on $T^{-2}\C/\F$.
\end{prop}

This canonical partial connection on $T^{-2}\C/\F$ can be expressed explicitly using
the data associated to a local non-vanishing section $f\in\cO(\F)$ in Proposition
\ref{prop_splitting_f}. Recall that, associated to $f$, we have obtained there an
isomorphism $T^{-2}\C/\F\cong\textrm{gr}_{-2}(T\C)\oplus\cV\cC$ as well as partial
connections on the two summands on the right hand side. We will express this
splitting by writing sections of $T^{-2}\C/\F$ as vectors $\binom{q}{v}_f$ with
$q\in\cO(\textrm{gr}_{-2}(T\C))$ and $v\in\cO(\cV\cC)$.

\begin{prop}\label{prop_Bott}
Suppose $\C\subset\P TM$ is a cone structure equipped with a characteristic conic
connection $\F$. Fix a local nowhere vanishing section $f\in\cO(\F)$ and denote by
$\nabla^f$ the induced partial connections on $\emph{gr}_{-2}(T\C)$ and on $\cV\cC$
as defined in Proposition \ref{prop_splitting_f}.  Then there is a local section
$\tilde\mu^f\in\cO(\emph{gr}_{-2}(T\C)^*\otimes\F^*\otimes\cV\cC)$ such that with
respect to the splitting
\begin{equation*}
T^{-2}\C/\F\cong\emph{gr}_{-2}(T\C)\oplus\cV\cC ,
\end{equation*}
induced by $\pi_f$ of Proposition
\ref{prop_splitting_f}, and in the notation from above, we get:
\begin{equation}\label{formula_Bott}
\nabla^\mathcal B_g \begin{pmatrix}q\\ v\end{pmatrix}_f= \begin{pmatrix}\nabla^f_gq
    +\mathcal L(g,v)\\ \nabla^f_gv-\tilde\mu^f(q,g)\end{pmatrix}_f
  \in\begin{matrix}\cO(\emph{gr}_{-2}(T\C))\\\oplus\\\cO(\cV\cC )\end{matrix},
\end{equation}
for any local section $g\in\cO(\F)$. Explicitly, for $q=\Cal L(f,v)$, we get
\begin{equation*}
\tilde\mu^f(q,g)=-\pi_f^{\cV\cC }([g,[f,v]-\nabla^f_fv]).
\end{equation*}
\end{prop}
\begin{proof}
We only need to verify \eqref{formula_Bott} for $g=f$ as \eqref{formula_Bott} is
tensorial in $g$. For $v\in\cO(\cV\cC )$ one has, by definition of the Bott
connection and the Levi-bracket,
$$\nabla_f^{\mathcal B}\begin{pmatrix}0\\ v\end{pmatrix}_f=\begin{pmatrix}q_{-2}([f,v])\\ \pi_f^{\cV\cC }([f,v])\end{pmatrix}_f=\begin{pmatrix}\mathcal L(f,v)\\ \nabla^f_fv\end{pmatrix}_f\in\cO(\textrm{gr}_{-2}(T\C))\oplus \cO(\cV\cC ),$$
since $\pi_f^{\cV\cC}([f,v])=\nabla_f^f v$ by Proposition \ref{prop_splitting_f} (4).
Now, under the isomorphism \eqref{direct_sum} induced by $\pi_f$, a section
$q\in\cO(\textrm{gr}_{-2}(T\C))$ corresponds to the unique section
$\xi\in \cO(T^{-2}\C)$ such that $q_{-2}(\xi)=q$ and $\pi_f(\xi)=0$.  By
\eqref{splitting_f}, $\xi=[f,v_q]-\nabla_f^fv_q$, where
$v_q=\phi_f^{-1}(q)\in\cO(\cV\cC )$. Hence the top component of
$\nabla_f^{\mathcal B}\binom{q}{0}_f$ is given by $q_{-2}([f,[f,v_q]-\nabla_f^fv_q])$
and it follows from (2) and (3) of Proposition \ref{prop_splitting_f} that
\begin{align*}
q_{-2}([f,[f,v_q]-\nabla_f^fv_q])&=q_{-2}([f,[f,v_q]])-\mathcal L(f,\nabla^f_f v_q)
=\mathcal L(f,\nabla^f_f v_q)\\
&=\nabla_f^f\mathcal L(f,v_q)=\nabla_f^f q.
\end{align*}
The bottom component of $\nabla_g^{\mathcal B}\binom{q}{0}_f$ for an arbitrary local section $g$ of $\F$  by definition is given by
\begin{equation}\label{formula_mu_f_2}
\pi_{f}^{\cV\cC }([g,[f,v_q]-\nabla_f^fv_q])=:-\tilde\mu^f(q,g).
\end{equation}
Since $\nabla^f$ and $\nabla^{\Cal B}$ are partial connections on $T^{-2}\C/\F\cong\textrm{gr}_{-2}(T\C)\oplus\cV\cC$ , the expression \eqref{formula_mu_f_2} has to depend
tensorially on both $q$ and $g$, which proves the claim.
\end{proof}

The tensorial nature of $\tilde\mu^f$ can be of course also easily verified directly from the explicit
formula \eqref{formula_mu_f_2} in the proposition, which is a nice sanity check. Now we can relate
$\tilde\mu^f$ to the cubic torsion of $\F$ as defined in \cite{H-N}:

\begin{prop}\label{tilde_rho} 
Let $\C\subset\P TM$ be a cone structure equipped with a characteristic conic
connection $\F$. For a local nowhere vanishing section $f\in\cO(\F)$, consider the locally defined
tensor $\tilde\mu^f$ from Proposition \ref{prop_Bott}. Then we have:
\begin{enumerate}
\item For $g\in\cO(\Cal F)$, the trace-free part
$\mu^f(g)\in\cO(\emph{Hom}_0(\emph{gr}_{-2}(T\C), \cV\cC ))$ of $\tilde\mu^f(g)$,
with respect to the isomorphism $\cL:\F\otimes\cV\cC \cong\emph{gr}_{-2}(T\C)$, is
independent of the choice of nowhere vanishing section $f$ and therefore defines an
invariant of the conic connection $\F$. 
\item Denoting the tensor corresponding to the trace-free part from (1) by
$\mu_\F\in\cO(\F^*\otimes \emph{Hom}_0(\emph{gr}_{-2}(T\C), \cV\cC ))$, then under
the isomorphism
\begin{align*}
 \F^*\otimes \emph{Hom}_0(\emph{gr}_{-2}(T\C),\cV\cC)&\cong  S^2\F^*\otimes\emph{Hom}_0(\cV\cC, \cV\cC)\\
 &\cong  S^3\F^*\otimes\emph{Hom}_0(\cV\cC, \emph{gr}_{-2}(T\C)),
 \end{align*}
it corresponds to the cubic torsion $\chi_\F$ as defined in \cite{H-N}.
\end{enumerate}
\end{prop}
\begin{proof}
(1) This follows from the construction of $\tilde\mu^f$ and does not need the
  explicit formula \eqref{formula_mu_f_2}.  For $q\in\mathcal{O}(\textrm{gr}_{-2}(T\C))$, the element
  $\xi=\binom{q}{0}_f$ is characterized by $q_{-2}(\xi)=q$ and
  $\pi_f^{\cV\cC}(\xi)=0$. Thus formula \eqref{change_pi_f} implies that for
  $\hat f=hf$, we get $\xi=\binom{q}{-\frac{1}{2h}dh(f)v}_{\hat f}$, where $q=\Cal
  L(f,v)$. Now \eqref{formula_Bott}, \eqref{change_pi_f} and the fact that
  $\nabla^f_g\Cal L(f,v)=\Cal L(f,\nabla^f_gv)$ show that
  $$
\nabla^{\Cal B}_g\begin{pmatrix} q \\ 0\end{pmatrix}_f=\begin{pmatrix} \nabla^f_g q
\\ -\tilde\mu^f(g,q) \end{pmatrix}_f=\begin{pmatrix} \nabla^f_g q
\\ -\tilde\mu^f(g,q)-\frac{1}{2h}dh(f)\nabla^f_g v \end{pmatrix}_{\hat f}.
$$
But this must be also equal to applying \eqref{formula_Bott} directly to $\binom{q}{-\frac{1}{2h}dh(f)v}_{\hat f}$,
which shows that the bottom component of the above expression has to equal
\begin{align*}
&-\nabla^{\hat f}_g(\frac{1}{2h}dh(f)v)-\tilde\mu^{\hat f}(g,q)=\\
&\quad\quad\quad=-g\cdot (\frac{1}{2h}df(h)) v-\frac{1}{2h}dh(f)\nabla^{\hat
  f}_g v-\tilde\mu^{\hat f}(g,q). 
\end{align*}
Using \eqref{change_nabla^f}, we conclude that
$\tilde\mu^f(g,q)$ and $\tilde\mu^{\hat f}(g,q)$ differ by some functional multiple of $v$,
which implies the claimed statement.

\smallskip
(2)
 This uses the explicit formula for $\tilde\mu^f$ from Proposition
\ref{prop_Bott}. For a local nowhere vanishing section $f\in\cO(\F)$, and $q=\Cal
L(f,v)\in \cO(\textrm{gr}_{-2}(T\C))$ we have
\begin{align*}
\tilde\mu^f(q,f)&=-\pi_f^{\cV\cC }([f,[f,v]-\nabla^f_fv]])\\
&\equiv -[f[f,v]]+2[f,\nabla_f^fv]-2\nabla_f^f\nabla_f^f v+\nabla_f^f\nabla_f^f v\mod \F\\
&\equiv  -[f[f,v]]+2[f,\nabla_f^fv]-\nabla_f^f\nabla_f^f v \mod \F.
\end{align*}
Hence, \eqref{def_nabla^f} implies
\begin{align*}
\mathcal L(f, \tilde\mu^f(q,f))&=-q_{-2}([f,[f,[f,v]]])+2q_{-2}([f,[f,\nabla_f^fv]])-q_{-2}([f,\nabla_f^f\nabla_f^f v])\\
&=-q_{-2}([f,[f,[f,v]]])+2q_{-2}([f,[f,\nabla_f^fv]])-\mathcal L(f,\nabla_f^f\nabla_f^f v)\\
&=-q_{-2}([f,[f,[f,v]]])+\frac{3}{2}q_{-2}([f,[f,\nabla_f^fv]]).
\end{align*}
This is exactly the formula used to define $\tilde\chi(f)(v)$ in Section 1.4 of \cite{H-N}, since
$\nabla_f^f v$ by definition coincides with the quantity $w(f,v)$ used there. 
\end{proof}

\begin{rem}
(1) Formally, the first part of Proposition \ref{tilde_rho} would not be really
  necessary, since it follows from the second part and the results in
  \cite{H-N}. Conceptually, it is very important, however, since it shows that
  considerations about the Bott connection and its relations to the structures
  induced by a local non-vanishing section of $\F$ provides a geometric way to obtain
  the cubic torsion (which was defined via a formula in \cite{H-N}) as an invariant
  of a characteristic conic connection.

  (2) In the setting of Proposition \ref{prop_Bott}, having chosen a local non-vanishing
  section $f\in\cO(\mathcal F)$, the maps $\pi_f$ and $\phi_f$ as defined in
  Proposition \ref{prop_splitting_f} induce also an isomorphism
\begin{equation*}
T^{-2}\C/\F\cong\textrm {gr}_{-2}(T\C)\oplus\cV\cC \cong \cV\cC \oplus\cV\cC .
\end{equation*}
With respect to this isomorphism, any section of $T^{-2}\C/\F$ can be written as
$\begin{pmatrix}w\\ v\end{pmatrix}_f$ for sections $v,w\in\cO(\cV\cC )$ and
  Propositions \ref{tilde_rho} and \ref{prop_Bott} imply that the Bott connection is
  given by
\begin{equation}\label{formula_Bott_2}
\nabla^\mathcal B_g \begin{pmatrix}w\\ v\end{pmatrix}_f=\begin{pmatrix}\nabla^f_g w
  +gf^{-1}v\\ \nabla^f_gv-gf^{-1}\phi_f^{-1}(\tilde\chi(f)(w))\end{pmatrix}_f\in\begin{matrix}\cO(\cV\cC
  )\\\oplus\\\cO(\cV\cC )\end{matrix},
\end{equation}
for any local section $g\in\cO(\F)$.
\end{rem}

Now we can derive a geometric interpretation for the vanishing of the cubic torsion of
a characteristic conic connection. 

\begin{thm}\label{geom_interpr_cubic_tor} 
  Suppose $\C\subset \P TM$ is a cone structure equipped with characteristic conic
  connection $\F$. Denote by $N$ a local leaf space of $\F$ and by
  $\mathcal D\subset TN$ the distribution induced by $T^{-2}\C/\F$.  The split
  quaternionic structure $\mathcal Q\subset \emph{End}(T^{-2}\C/\F)$ on $T^{-2}\C/\F$
  defined as in Proposition \ref{para-q-str} descends to a split quaternionic
  structure $\mathcal Q_{\mathcal D}\subset \emph{End}(\mathcal D)$ on $\mathcal D$
  if and only if $\chi_\F=0$. In this case, the distribution $\mathcal D$ inherits a
  canonical Segre structure and hence a natural local isomorphisms to a tensor
  product of auxiliary bundles of rank $2$ and $\emph{rank}(\cV\cC)$, respectively.
\end{thm}
\begin{proof}
  We will show that the Bott connection $\nabla^{\mathcal B}$ preserves $\mathcal Q$
  if and only if $\chi_\F=0$, which proves the claim. Suppose $f$ is a local
  non-vanishing section of $\F$. With respect to the induced isomorphism
  $T^{-2}\C/\F\cong \cV\cC \oplus\cV\cC $, the endomorphisms $I_f$, $J_f$ and
  $K_f=I_f\circ J_f$ correspond to
\begin{equation*}
I_f=\begin{pmatrix} -1 &0\\
0&1
\end{pmatrix}, \quad J_f=\begin{pmatrix} 0 &1\\
1&0
\end{pmatrix}, \, \textrm{ and }\, K_f=\begin{pmatrix} 0 &-1\\
1&0
\end{pmatrix}. 
\end{equation*}
By \eqref{formula_Bott_2}, we therefore have 
\begin{align*}
(\nabla_f^{\mathcal B} I_f)\begin{pmatrix}w\\ v\end{pmatrix}_f&=\begin{pmatrix}-\nabla^f_f w +v\\ \nabla^f_fv+\phi_f^{-1}(\tilde\chi(f)(w))\end{pmatrix}_f-\begin{pmatrix}-\nabla^f_f w -v\\ \nabla^f_fv-\phi_f^{-1}(\tilde\chi(f)(w))\end{pmatrix}_f\\
&=2\begin{pmatrix} v\\ \phi_f^{-1}(\tilde\chi(f)(w))\end{pmatrix}_f
\end{align*}
\begin{align*}
(\nabla_f^{\mathcal B} J_f)\begin{pmatrix}w\\ v\end{pmatrix}_f&=\begin{pmatrix}\nabla^f_f v +w\\ \nabla^f_fw-\phi_f^{-1}(\tilde\chi(f)(v))\end{pmatrix}_f-\begin{pmatrix}\nabla^f_fv-\phi_f^{-1}(\tilde\chi(f)(w)\\ \nabla_f^f w+v\end{pmatrix}_f\\
&=\begin{pmatrix} w+\phi_f^{-1}(\tilde\chi(f)(w))\\ -v-\phi_f^{-1}(\tilde\chi(f)(v))\end{pmatrix}_f
\end{align*}
\begin{align*}
(\nabla_f^{\mathcal B} K_f)\begin{pmatrix}w\\ v\end{pmatrix}_f&=\begin{pmatrix}-\nabla^f_f v +w\\ \nabla^f_fw+\phi_f^{-1}(\tilde\chi(f)(v))\end{pmatrix}_f-\begin{pmatrix}-\nabla^f_fv+\phi_f^{-1}(\tilde\chi(f)(w)\\ \nabla_f^f w+v\end{pmatrix}_f\\
&=\begin{pmatrix} w-\phi_f^{-1}(\tilde\chi(f)(w))\\ -v+\phi_f^{-1}(\tilde\chi(f)(v))\end{pmatrix}_f
\end{align*}

If $\chi_\F=0$, then $\tilde\chi(f):\cV\cC \rightarrow\textrm{gr}_{-2}(T\C)$ is pure
trace, which means it is a functional multiple of
$\phi_f:\cV\cC \rightarrow\textrm{gr}_{-2}(T\C)$, and hence
$\nabla_f^{\mathcal B} I_f$, $\nabla_f^{\mathcal B} J_f$ and
$\nabla_f^{\mathcal B} K_f$ are evidently in the span of $I_f$, $J_f$ and
$K_f$. Conversely, assume $\nabla_f^{\mathcal B} I_f$, $\nabla_f^{\mathcal B} J_f$
and $\nabla_f^{\mathcal B} K_f$ are in the span of $I_f$, $J_f$ and $K_f$. In
particular, there must be functions $a,b,c$ such that
\begin{equation*}
  (\nabla_f^{\mathcal B} I_f)\begin{pmatrix}w\\ 0\end{pmatrix}_f=
\begin{pmatrix} 0\\ 2\phi_f^{-1}(\tilde\chi(f)(w))\end{pmatrix}_f=
\begin{pmatrix} -a &b\\
c&a
\end{pmatrix}
\begin{pmatrix} w\\
0
\end{pmatrix}_f=\begin{pmatrix} -aw\\
cw
\end{pmatrix}_f,
\end{equation*}
for any section $w\in\cO(\cV\cC )$, which implies that $\tilde\chi(f):\cV\cC
\rightarrow\textrm{gr}_{-2}(T\C)$ has to be pure trace.
\end{proof}

\begin{rem}
Note that all the notions and results of Sections 2 hold also in the smooth category,
i.e. for smooth cone structures on smooth manifolds and smooth characteristic conic
connections.
\end{rem}

\section{Conic connections induced by principal connections}
To any isotrivial cone structure, as defined in Definition \ref{def_cone_str} (3),
one can associate a first order $G$-structure and any principal connection on this
$G$-structure induces a conic connection on the given cone structure
\cite{Hwang-Li-2,H-N}.  In this section we will recall that relation and study, in
the case of isotrivial cone structures with homogeneous fibers, how the
characteristic and cubic torsion of a conic connection that is induced from a
connection on the associated $G$-structure is related to the two basic local
invariants of such a connection, namely its torsion and curvature.
\subsection{The associated $G$-structure of an isotrivial cone structure}\label{sec_assoc_G-structure}
Let $\C\subset\P TM$ be a $Z$-isotrivial cone structure on a complex manifold $M$ for
a closed complex submanifold $Z\subset\P W$, where $W$ is a vector space with
$\dim(W)=\dim(M)$, see Definition \ref{def_cone_str} (3). To such a structure, we can
associate a first-order $G$-structure on $M$ with structure group
$$
\textrm{Aut}(\widehat Z)=\{g\in\textrm{GL}(W): g\widehat Z=\widehat
Z\}\subset\textrm{GL}(W)$$
as follows, cf.\,\cite{Hwang-Li-2,H-N}: we can view the frame bundle $\pi:
FM\rightarrow M$ of $M$ as a principal $\textrm{GL}(W)$-bundle, where the fiber over
a point $x\in M$ is given by
 \begin{equation}\label{fiber_frame_bdl}
 F_xM=W^*\otimes T_xM=\textrm{Hom}(W,T_xM),
 \end{equation}
 on which $\textrm{GL}(W)$ acts from the right by pre-composition.  For $x\in M$,
 set $$\mathcal P_x:=\{\phi\in\textrm{Hom}(W,T_xM): \phi(\widehat
 Z)=\widehat\C_x\}\subset F_xM.$$ Then the disjoint union
\begin{equation}\label{assoc_G-structure}
\mathcal P:=\bigsqcup_{x\in M} \mathcal P_x\subset FM
\end{equation}
defines a principal $\textrm{Aut}(\widehat Z)$-subbundle of $FM$, that is, a
reduction of structure group of the frame bundle of $M$ to the subgroup
$\textrm{Aut}(\widehat Z)\subset\textrm{GL}(W)$.
\begin{defin} Suppose $\C\subset \mathbb P TM$ is a $Z$-isotrivial cone structure on
  $M$ for some closed submanifold $Z\subset \mathbb P W$ and set
  $G:=\textrm{Aut}(\widehat Z)\subset\textrm{GL}(W)$. Then
\[
\begin{tikzcd}
    \mathcal P \arrow[hookrightarrow]{rr}{\iota} \arrow[swap]{dr}{} & & FM \arrow{dl}{\pi} \\
    & M 
\end{tikzcd}
\]
 as defined in \eqref{assoc_G-structure} is referred to as the $G$-structure
 associated to $\mathcal C\subset \mathbb PTM$. By its definition, $\C$ and $\mathbb
 P TM$ are then associated bundles
 \begin{equation}\label{assoc_cone}
 \C\cong \mathcal P\times_G Z\subset \mathcal P\times_G \mathbb P W\cong \mathbb
 PTM
 \end{equation}
 to the principal $G$-bundle $\mathcal P$ corresponding to the action of $G$ on
 $\mathbb P W$ respectively to its restriction to $Z$.
\end{defin}

Conversely, having fixed a subvariety $Z\subset \mathbb P W$, any $G$-structure
$\mathcal P\subset FM$ with structure group
$\textrm{Aut}(\widehat Z)\subset \textrm{GL}(W)$ induces an $Z$-isotrivial cone
structure by forming the associate bundle as in \eqref{assoc_cone}.  Observe
that different subvarieties of $\mathbb PW$ may lead to the same subgroup of
automorphisms in $\textrm{GL}(W)$. Once $Z\subset\P W$ is fixed, however, there is a
natural bijective correspondence between $G$-structures with structure group
$\textrm{Aut}(\widehat Z)$ and $Z$-isotrivial cone structures.

\subsection{Review of connections on $G$-structures}\label{rev_princ_conn}
Suppose $M$ is a manifold and $W$ a vector space of dimension $\dim M$. View the frame bundle $\pi:FM\rightarrow M$ as being modeled on $W$ as in \eqref{fiber_frame_bdl}, which 
is encoded by a $W$-valued $1$-form $\Theta\in\Omega^1(FM, W)$, given by
$$\Theta(\xi_u)=u^{-1}(T_u\pi \xi_u)\quad\forall \xi_u\in T_uFM,$$
called the \emph{soldering form}. It is $\textrm{GL}(W)$-equivariant and
$\ker(\Theta(u))=\ker(T_u\pi)$ for all $u\in FM$.

Let $G\subset \textrm{GL}(W)$ be a closed subgroup and
$\iota:\mathcal P\hookrightarrow FM$ be a reduction of structure group of the frame
bundle $FM$ to $G$.  Denote by $$\theta:=\iota^*\Theta\in\Omega^1(\cP, W)$$ the
soldering form on $\cP$, which, by construction, is $G$-equivariant and satisfies
$\ker(\theta(u))=\ker(T_u(\pi\circ\iota))$ for all $u\in \cP$.  Moreover, we will
denote by $\mathcal V\mathcal P\subset T P$ the vertical bundle of
$\mathcal P\rightarrow M$ and for any $X\in \g$, the Lie algebra of $G$, by
$\zeta_X\in\mathfrak X(\cP)$ the vertical vector field generated by $X$, so
$\zeta_X(u):=\frac{d}{dt}\vert_{t=0} u\cdot\exp(tX)$ for any $u\in\mathcal P$. Recall
that mapping $(u,X)\in\cP\times\mathfrak g$ to $\zeta_X(u)$ defines a canonical
trivialization
\begin{equation}\label{trivialization_VP}
\mathcal P\times \g\cong \cV\cP
\end{equation}
of the vertical bundle $\cV\cP\rightarrow\cP$.

\begin{defin} A \emph{connection on a $G$-structure} $\iota:\cP\hookrightarrow FM$ is
  a principal connection on the principal $G$-bundle $\mathcal P$, i.e.\  
a (horizontal) distribution $\mathcal H\mathcal P\subset T\mathcal P$ on $\mathcal
P$ complimentary to $\cV\cP$,
\begin{equation}\label{principal_conn}
T\mathcal P=\mathcal V\mathcal P\oplus \mathcal H\mathcal P,
\end{equation}
such that $T_ur^g (\mathcal H_u\mathcal P)= \mathcal H_{ug}\mathcal P$ for all $u\in
P$ and $g\in G$, where $r^g:\cP\rightarrow \cP$ denotes the principal right-action of
$g\in G$ on $\cP$.
\end{defin}
Via \eqref{trivialization_VP}, the vertical projection $T\cP\rightarrow\cV\cP$
encoding a principal connection as in \eqref{principal_conn}, can be seen as
$G$-equivariant $\g$-valued $1$-form $\gamma\in\Omega^1(\cP,\g)$ such that
$\gamma(\zeta_X)=X$ for all $X\in\g$.

Recall that for a principal $G$-bundle $\cP\to M$ and a $G$-action on a smooth
manifold $S$, one can form the associated fiber bundle $\cP\times_G S\rightarrow M$
with typical fiber $S$. Starting from a representation of $G$ on a vector space $E$,
the bundle $\mathcal {E}:=\cP\times_G E\to M$ naturally is a vector bundle. It is
well known that a principal connection $\gamma$ on $\cP$ induces a linear connection
on $\mathcal E$. This linear connection can be equivalently encoded by a horizontal
subbundle in $T\mathcal E$ which turns out to be the image of the horizontal
subbundle $\Cal H\cP\subset T\cP$ under the natural map $\cP\x E\to \mathcal E$. The
latter point of view extends to general associated bundles, so on $\cP\x_G S$ one
obtains a horizontal distribution which is complementary to the vertical
subbundle. Hence, one obtains what is called an \emph{Ehresmann connection} on the
fiber bundle $\cP\x_G S\to M$.

For a $G$-structure $\cP\hookrightarrow F M$ corresponding to $G\subset \textrm{GL}(W)$, the associated vector bundle
$\cP\x_G W$ by definition is the tangent bundle $TM$. Hence, any connection $\gamma$ on $\cP$ induces a linear connection on $TM$ and so it also specifies a
distinguished class of curves on $M$, given by the geodesics of this induced
connection. We refer to them as the \emph{geodesics of $\gamma$}.

For later use, let us also remark, that a connection $\gamma\in\Omega^1(\cP, \g)$ on
a $G$-structure $\cP\hookrightarrow F M$ can be naturally extended via the
soldering form to a \emph{Cartan connection} on $\cP$ of type $(\tilde G, G)$, where
$\tilde G:=W\rtimes G$ it the \textit{affine extension} of $G$:  Indeed, $\tilde
\g=W\oplus\g$ as a $G$-representation and $G$-equivariancy of $\theta$ and $\gamma$
implies $G$-equivariancy of
\begin{equation*}
\omega=\theta+\gamma\in\Omega^1(\cP, W\rtimes \g).
\end{equation*}
Moreover, $\ker(\theta(u))=\cV_u\cP$ implies that $\omega(\zeta_X)=\gamma(\zeta_X)=X$
for any $X\in\g$ and that $\omega(u): T_u\cP\rightarrow\tilde \g$ is an isomorphism
for all $u\in \cP$.  Therefore,  $\omega$ is indeed a Cartan connection of type
$(\tilde G,G)$ on $\cP$. The \emph{curvature of the Cartan connection} $\omega$ is given by
the $G$-equivariant $\tilde \g$-valued $2$-form
\begin{equation}\label{def_Cartan_curv}
K(\xi,\eta)=d\omega(\xi,\eta)+[\omega(\xi),\omega(\eta)] \quad\textrm{ for }
\xi,\eta\in\cO(T\cP).
\end{equation}
By $G$-equivariancy, we may decompose $K$ into an $W$-valued component, which equals the torsion of $\gamma$, given by
\begin{equation}\label{torsion_gamma}
\tau(\xi,\eta)=d\theta(\xi,\eta)+\gamma(\xi)(\theta(\eta))-\gamma(\eta)(\theta(\xi)),
\end{equation}
and a $\g$-valued component, which equals the curvature of $\gamma$,
given by
\begin{equation}\label{curv_gamma}
\rho(\xi,\eta)=d\gamma(\xi,\eta)+[\gamma(\xi),\gamma(\eta)]. 
\end{equation}
The defining properties of a Cartan connection easily imply that the curvature form
$K$ is horizontal and $G$-equivariant, hence the same is true for $\tau$ and
$\rho$. Therefore, these forms descend to two-forms on $M$ with values in the
associated bundles $\cP\x_GW=TM$ and $\cP\x_G\g$, respectively. Viewed as an element
of $\Omega^2(M,TM)$, $\tau$ is the torsion of the induced linear connection on
$TM$. For the other component, we observe that $\g\subset \textrm{Hom}(W,W)$ so
$\rho$ descends to a two-form with values in $\textrm{Hom}(TM,TM)$ and in this
interpretation it is the curvature of the induced linear connection on $TM$.

\subsection{Conic connections induced by principal connections}
Let $\C\subset \mathbb P TM$ be an $Z$-isotrivial cone structure for some $Z\subset
\mathbb P W$ and $\iota:\mathcal P\hookrightarrow FM$ the associated $G$-structure as
defined in Section \ref{sec_assoc_G-structure}.  Suppose
$\gamma\in\Omega^1(\cP,\mathfrak g)$ is a principal connection on $\cP$ with
corresponding horizontal distribution $\cH\cP$. Then, as discussed in Section \ref{rev_princ_conn}, $\gamma$ induces an Ehresmann
connection on $\cC\cong \cP\times _GZ$, that is, a decomposition
\begin{equation}
T\cC=\cV\cC\oplus\cH\cC.
\end{equation}
In particular, on the subbundles $T^{-1}\C\subset T^{-2}\C\subset T\C$ as defined in \eqref{natural_distributions_C}, we get induced decompositions
\begin{equation}
T^{-1}\cC=\cV\cC\oplus \cH^{-1}\cC\quad\textrm{ and }\quad
T^{-2}\cC=\cV\cC\oplus \cH^{-2}\cC,
\end{equation}
where $\cH^{-i}\cC:=\cH\cC\cap T^{-i}\cC$ for $i=1,2$. Note that
$\F^\gamma:=\cH^{-1}\cC$ defines a conic connection on $\cC$, which we will refer to
as the \emph{conic connection on $\cC$ induced by $\gamma$}. It corresponds to the
$\gamma$-geodesics in directions that lie in $\cC\subset\P TM$ (which stay in $\cC$
by construction).

\subsection{Isotrivial cone structures with homogeneous fibers}
We now specialize to the case that $G:=\textrm{Aut}(\widehat Z)$ acts transitively on
$Z$. So let $\cC\subset \P TM$ be a $Z$-isotrivial cone structure on a manifold $M$,
where $G$ acts transitively on $Z$. Fixing a base point $z_0\in Z$, we get an
identification $Z\cong G/H$, where $H\subset G$ is the stabilizer of $z_0$. Denoting
by $(\pi_M:\cP\rightarrow M,\theta)$ the associated $G$-structure, we have that
$$\cC\cong\cP\times_G Z\cong \cP\times_G G/H=\cP/H,$$ so the natural projection
$\pi_\cC: \cP\rightarrow \cP/H=\cC$ makes $\cP$ a principal $H$-bundle over $\cC$
such that $p\circ \pi_\cC=\pi_M$. This brings us into the general setting of
correspondence spaces, see \cite{Cap_Corresp}. In particular, one easily verifies:

\begin{prop}\label{Corr_space_prop} 
Suppose that $\gamma\in\Omega^1(\cP,\mathfrak g)$ is a connection on the
$G$-structure $(\pi_M: \cP\rightarrow M, \theta)$ and denote by
$\omega=\theta+\gamma\in\Omega^{1}(\cP, \tilde \g)$ the associated Cartan connection
of type $(\tilde G=W\rtimes G, G)$.
\begin{enumerate}
\item Then $\omega$ is a Cartan connection of type $(\tilde G, H)$ on $\pi_\cC:
  \cP\rightarrow \cP/H=\cC$. So $\cC$ carries a canonical Cartan geometry, which is
  called the correspondence space of the Cartan geometry $(\cP\rightarrow M, \omega)$
  with respect to $H\leq G$ in the terminology of \cite{Cap_Corresp}.
 \item Via the two interpretations as a Cartan connection, $\omega$ induces vector
   bundle isomorphisms
 \begin{equation}\label{omega_ind_isos}
T\cC\cong \cP\times_H \tilde \g/\h\quad\textrm{ and }\quad TM\cong \cP\times_G \tilde \g/\g\cong\cP\times_G W,
\end{equation}
where the second isomorphism is independent of $\gamma$ but only depends on $\theta$.
 \end{enumerate}
 Moreover, under the isomorphisms \eqref{omega_ind_isos}, the tangent map
 $Tp:T\cC\rightarrow TM$ of $p:\cC\rightarrow M$ corresponds to the natural
 projection $\tilde \g/\h\to \tilde \g/\g$. In particular, we have $\cV\cC \cong
 \cP\times_H \g/\h$.
\end{prop}

In the setting of Proposition \ref{Corr_space_prop}, since
 $\tilde \g\cong W\oplus\g$ as $G$-module, in particular as $H$-module, we have
\begin{equation}\label{decom_rep}
 \tilde \g/\h\cong W\oplus\g/\h 
 \end{equation}
 as $H$-module. Via the first isomorphism in \eqref{omega_ind_isos}, $\cP\times_H W$ defines a
 subbundle of $T\cC$ complimentary to $\cV\cC$, which equals
   
\begin{equation}\label{HC=W}
 \mathcal H\cC\cong\cP\times_H W\subset T\cC,
 \end{equation}
 by definition of the Ehresmann connection $\mathcal H\cC$ on $p:\cC\rightarrow M$
 induced by $\gamma$. Moreover, setting  $W^{-1}:=\hat z_0$ and $W^{-2}:=\widehat T_{z_0} Z=[\g, \hat z_0]$, defines
 an $H$-invariant filtration 
  \begin{equation}\label{filtration_W}
 W^{-1}\subset W^{-2}\subset W,
 \end{equation}
 which is the beginning of the osculating filtration of $Z\subset \mathbb P W$
 determined by $z_0\in Z$ as briefly discussed in Remark \ref{comparison_filt_rem}.
 For later use, we also set $\tilde \g^{-i}:=W^{-i}\oplus\g$ for $i=1,2$ and
 $\tilde \g^0:=\h$, yielding an $H$-invariant filtration
 \begin{equation}\label{filtration_tilde g}
 \tilde\g^0\subset \tilde\g^{-1}\subset \tilde\g^{-2}\subset\tilde\g,
 \end{equation}
 and an $H$-invariant filtration on $\tilde \g/\h=W\oplus \g/\h$. Also, note that the Lie bracket on $\tilde \g$ respectively the action of $\g$ on $W$, gives rise to an 
 $H$-equivariant bilinear map
 \begin{equation}\label{alg_bracket}
 \g/\h\times W^{-1}\rightarrow W^{-2}/W^{-1}.
\end{equation}

 \begin{lem} The first isomorphism in \eqref{omega_ind_isos} identifies
 \begin{equation*}
T^{-1}\C= \cH^{-1}\C\oplus\cV\cC=\F^\gamma\oplus\cV\cC\cong \cP\times_H (W^{-1}\oplus \g/\h)
 \end{equation*}
  \begin{equation*}
T^{-2}\C= \cH^{-2}\C\oplus\cV\cC\cong \cP\times_H (W^{-2}\oplus \g/\h).
 \end{equation*}
 Moreover, the map  
 \begin{equation}\label{algebraic_bracket}
\{_-,_-\}: \cV\cC\otimes \F^\gamma \rightarrow  \emph{gr}_{-2}(T\C),
\end{equation}
induced by \eqref{alg_bracket} coincides with the Levi bracket as defined in
\eqref{Levi_iso}.
 \end{lem}
 \begin{proof}
 The first statement follows directly from the description of $T^{-1}\C$ and $T^{-2}\C$ and the fact that,
 under the isomorphisms \eqref{omega_ind_isos}, $Tp:T\cC\rightarrow TM$
 corresponds exactly to the natural projection $\tilde \g/\h\to \tilde \g/\g$. The preimages under $T\pi_\C: T\mathcal P\rightarrow T\C$ 
 of the distributions $\cV\C$ and $\F^\gamma$ correspond to $\omega^{-1}(\g)$ and $\omega^{-1}(W^{-1}\oplus\h)$ respectively and
 locally sections of $\cV\C$ and $\F^\gamma$ lift to sections of $\omega^{-1}(\g)$ respectively $\omega^{-1}(W^{-1}\oplus\h)$.
 Since $\omega$ is a Cartan connection of type $(\tilde G,G)$, horizontality of the Cartan curvature implies that
 for sections $\xi$ of $\omega^{-1}(\g)$ and $\eta$ of $\omega^{-1}(W^{-1}\oplus\h)$, one has
 \begin{align*}
 0=K(\xi,\eta)&=\xi\cdot\omega(\eta)-\eta\cdot\omega(\xi)-\omega([\xi,\eta])+[\omega(\xi),\omega(\eta)]\\
 &\equiv -\omega([\xi,\eta])+[\omega(\xi),\omega(\eta)]\mod W^{-1}\oplus\g.
 \end{align*}
 As $\omega([\xi,\eta])\mod W^{-1}\oplus\g$ describes the Levi-bracket $\mathcal L:\cV\cC\otimes \F^{\gamma}\rightarrow\textrm{gr}_{-2}(T\C)$ and $[\omega(\xi),\omega(\eta)]\mod W^{-1}\oplus\g$ the bracket \eqref{algebraic_bracket}, 
 this proves the second claim.
 \end{proof}
 
 \begin{prop}\label{prop_part_conn} 
   Suppose that $\gamma\in\Omega^1(\cP,\mathfrak g)$ is a connection on the
   $G$-structure $(\pi_M: \cP\rightarrow M, \theta)$ and denote by
   $\omega=\theta+\gamma\in\Omega^{1}(\cP, \tilde \g)$ the associated Cartan
   connection.  For any $H$-module $E$ the Cartan connection $\omega=\theta+\gamma$
   induces a partial linear $\mathcal H\cC$-connection
 \begin{equation}\label{gamma_ind_partial_conn}
 \nabla^\gamma: \cO(\mathcal H\mathcal C)\times \cO(\mathcal E)\rightarrow \cO(\mathcal E).
 \end{equation}
 on the associated vector bundle $\mathcal E=\cP\times_H E\rightarrow \cC$. In
 particular, we have induced partial connections on $\cV\cC $, $\F^\gamma$,
 $\cH^{-2}\cC$ and
 $\emph{gr}_{-2}(T\cC)=\emph{gr}_{-2}(\cH\cC)=\cH^{-2}\cC/\F^\gamma$. Moreover, the
 bracket
 $\{_-,_-\}=\mathcal L: \cV\cC\otimes \F^\gamma\rightarrow \emph{gr}_{-2}(T\cC)$ is
 parallel with respect to $\nabla^\gamma$.
 \end{prop}
\begin{proof}
 By construction, for each $u\in\cP$, the tangent map $T_u\pi_{\cC}$ restricts to a
 linear isomorphism $\cH_u\cP\to \cH_{\pi_{\cC}(u)}\cC$. Hence we can lift tangent
 vectors from $\cC$ to $\cP$, provided that they lie in the subbundle
 $\cH\cC\subset T\cC$. This also allows us to lift a local section
 $\xi\in\cO(\cH\cC)$ to a local section $\xi^h\in\cO(\cH\cP)$. Since the
 subbundle $\cH\cP\subset T\cP$ by definition is $G$-invariant and hence
 $H$-invariant, it follows readily that $\xi^h$ is an $H$-invariant vector field on
 $\cP$. Thus we can imitate the usual construction of induced linear connections. For
 $s\in\mathcal{O}(\cE)$ consider the corresponding $H$-equivariant function $f_s:\cP\to
 E$. Then also $\xi^h\cdot f_s$ is equivariant and we define $\nabla^\gamma_\xi f$ to
 be the corresponding section.

One verifies straightforwardly that $\nabla^\gamma$ defines a partial connection as
 claimed. Since $\{_-,_-\}$ is induced by an $H$-equivariant map, it corresponds to a
 constant function $\cP\rightarrow \textrm{Hom}(\g/\h\otimes W^{-1}, W^{-2}/W^{-1})$
 and hence is parallel for $\nabla^\gamma$.
 \end{proof}
 
 We will now relate the characteristic and cubic torsion of $\F^\gamma$ to the
 Cartan curvature $K\in\Omega^2(\cP,\tilde\g)$ of $\omega$.
 Recall that $K=\tau+\rho$, where $\tau\in\Omega^2(\cP, W)$ is the torsion and
 $\rho\in\Omega^2(\cP,\g)$ is the curvature of $\gamma$.  Since $K$ and its
 components $\tau$ and $\rho$ are $G$-equivariant, in particular $H$-equivariant, and
 horizontal for the projection $\pi_M:\cP\rightarrow M$ and hence for
 $\pi_\C:\cP\rightarrow \C$, they descend to sections
 \begin{equation}
 K^M:=\tau^M+\rho^M\in\Omega^2(M,TM\oplus \cP\times_G\g)
 \end{equation}
 as well as to sections
 \begin{equation}
 K^\C:=\tau^\C+\rho^\cC\in\Omega^2(\C, \cH\C\oplus\cP\times_H\g),
 \end{equation}
 which vanish upon insertion of elements of $\cV\cC$.  The following proposition
 relates the characteristic torsion of $\F^\gamma$ to the torsion of $\gamma$. This
 recovers a result of \cite[Theorem 3.8]{Hwang-Li-2} in the case of isotrivial cone
 structures with homogeneous fibers:
 
\begin{thm}\label{thm_characteristic_torsion} 
Suppose that $\gamma\in\Omega^1(\cP,\mathfrak g)$ is a connection on the
$G$-structure $(\pi_M: \cP\rightarrow M, \theta)$ with associated Cartan connection
$\omega=\theta+\gamma\in\Omega^{1}(\cP, \tilde \g)$ and let
$\F:=\F^\gamma\cong\cP\times_HW^{-1}\subset T\cC$ be the induced conic connection on
$\cC\subset \mathbb P TM$. Then the following is equivalent:
\begin{enumerate}
\item $\F$ is characteristic
\item $\tau^\cC(f,\xi)\in \cO(T^{-2}\C)$ for all $f\in\cO(\F)$ and $\xi\in \cO(T^{-2}\C)$
\item $\tau^\cC(f,\xi)\in \cO(\cH^{-2}\C)$ for all $f\in\cO(\F)$ and $\xi\in \cO(\cH^{-2}\C)$
\item $\tau^M$ is a section of the subbundle $\cP\times_G \Xi_Z\subset \Wedge^2 T^*M\otimes TM$, where
\begin{equation}\label{Xi}
\Xi_Z:=\{\phi\in\Wedge^2 W^*\otimes W: \phi(\hat z, \wh T_{z}Z)\subset\wh T_{z}Z)\quad\forall z\in Z\}.
\end{equation}
\end{enumerate}
\end{thm}
 \begin{proof}
   The equivalence of (2) and (3) is clear, since $\tau^\cC$ by definition has values
   in $\cH\C$ and vanishes upon insertion of a section of $\cV\cC$.  To see that (1)
   is equivalent to (3), note first that $\F$ is characteristic if and only if the
   Lie bracket of a section of $\F$ with a section of $\cH^{-2}\cC$ is a section of
   $T^{-2}\C$.  Suppose now that $f$ is a section of $\F\cong\cP\times_H W^{-1}$ and
   $\xi$ is a section of $\cH^{-2}\cC\cong\cP\times_H W^{-2}$ and let
   $f^h, \xi^h\in\mathfrak X(\cP)^H$ be their horizontal lifts to $\cP$ as in the
   proof of Proposition \ref{prop_part_conn}.  Then $[f^h,\xi^h]$ is a lift of
   $[f,\xi]$ and $[f,\xi]$ is a section of
   $T^{-2}\cC\cong \cP\times_H\tilde \g^{-2}/\h$ if and only if
   $\omega([f^h,\xi^h]):\cP\rightarrow \tilde \g$ has values in
   $\tilde \g^{-2}=W^{-2}\oplus\g$.  By definition of the curvature, we have
 \begin{align}\label{bracket_F and H2}
 \omega([f^h,\xi^h])&=-K(f^h,\xi^h)+[\omega(f^h),\omega(\xi^h)]+f^h\cdot\omega(\xi^h)-\xi^h\cdot\omega(f^h)\\\nonumber
 &=-\rho(f^h,\xi^h)-\tau(f^h,\xi^h)+f^h\cdot\omega(\xi^h)-\xi^h\cdot\omega(f^h),
 \end{align}
 where the second equality follows from $W\subset \tilde \g$ being an abelian
 subalgebra. Since $\rho(f^h,\xi^h)$, $f^h\cdot\omega(\xi^h)$, and
 $\xi^h\cdot\omega(f^h)$ have values in $\g$, $W^{-2}$, and $W^{-1}$, respectively,
 we see that $\omega([f^h,\xi^h])$ has values in $\tilde\g^{-2}$ if and only if
 $\tau(f^h,\xi^h)$ has values in $W^{-2}$. By definition of $\tau^\C$, this is
 equivalent to $\tau^\cC(f,\xi)$ having values in $\cH^{-2}\cC$.  Hence, (1) is
 equivalent to (3).

 To see that (3) is equivalent to (4), consider the $H$-invariant
 subspace
 $$\Xi_Z^H:=\{\phi\in\Wedge^2 W^*\otimes W: \phi(W^{-1}, W^{-2})\subset
 W^{-2}\}\subset \Wedge^2 W^*\otimes W.$$ By definition, $\tau^\C$ and $\tau^M$ are
 both induced by $\tau$, which can be identified with a $G$-equivariant function
 $\cP\rightarrow \Wedge^2W^*\otimes W$. Now, (3) is equivalent to $\tau$, viewed as
 an $H$-equivariant function, having values in $\Xi_Z^H$. However, since $\tau$ is
 also a $G$-equivariant function, giving rise to a section $\tau^M$, it must have
 values in a $G$-invariant subspace of $\Xi_Z^H$. Since, for a different choice of
 base point $z'_0=g\cdot z_0\in Z$, the filtration \eqref{filtration_W} changes by
 multiplication by $g$, i.e. $g\cdot \widehat T_{z_0}Z=\widehat T_{g\cdot z_0}Z$, the
 largest such subspace is $\Xi_Z\subset\Xi_Z^H$. This implies that (3) is equivalent
 to (4).
 \end{proof}
 
 Assuming that $\F^\gamma$ is characteristic, we want to relate its cubic torsion of
 $\F^\gamma$ to the curvature of $\omega$. As a first step, we compute the Bott
 connection in terms of $\nabla^\gamma$.
 
 \begin{thm}\label{Thm_Bott_in_terms_of_gamma} 
In the setting of Proposition \ref{prop_part_conn} assume that $\F:=\F^\gamma$ is
characteristic. Then in terms of the splitting
 \begin{equation}\label{gamma_ind_splitting}
 T^{-2}\C/\F\cong\emph{gr}_{-2}(\cH\cC)\oplus \cV\cC=\emph{gr}_{-2}(T\cC)\oplus \cV\cC,
\end{equation} for which we use a vector notation with subscript $\ga$, 
and the partial connection $\nabla^\gamma$ induced by $\gamma$, the Bott connection is given by
 \begin{equation}\label{formula_Bott_in_terms_of_gamma}
\nabla^\mathcal
B_f \begin{pmatrix}q\\ v\end{pmatrix}_\ga= \begin{pmatrix}\nabla^\gamma_fq
    -\tau_{-2}^\cC(f,q)+\mathcal
    L(f,v)\\ \nabla^\gamma_fv-\rho^\cC_{-1}(f,q)\end{pmatrix}_\ga
  \in\begin{matrix}\cO(\emph{gr}_{-2}(T\C))\\\oplus\\\cO(\cV\cC)\end{matrix},
\end{equation}
for sections $f\in\cO(\F)$ , $q\in\cO(\emph{gr}_{-2}(T\cC))$ and
$v\in\cO(\cV\cC)$. Here, $\tau_{-2}^\cC(f,q)$ denotes the projection of
$\tau^\cC(f,q)\in\cO(\cH^{-2}\cC)$ to
$\cO(\emph{gr}_{-2}(\cH\cC))=\cO(\emph{gr}_{-2}(T\cC))$ and
$\rho^\cC_{-1}(f,q)$ the projection of $\rho^\cC(f,q)\in\cO(\cP\times_H\g)$ to
$\cO(\cV\C)=\cO(\cP\times_H\g/\h)$.
 \end{thm}
 \begin{proof}
 Suppose that $f\in\cO(\F)$ and $\xi\in\mathcal{O}(\cH^{-2}\cC)$ and let $f^h,
 \xi^h\in\mathfrak X(\cP)^H$ be their horizontal lifts to $\cP$ as in the proof of
 Theorem \ref{thm_characteristic_torsion}. By Theorem
 \ref{thm_characteristic_torsion}, the expression \eqref{bracket_F and H2} has values
 in $\tilde \g^{-2}=W^{-2}\oplus\g$ and implies that
 
 \begin{equation}\label{formula_Bott_in_terms_of_gamma_1}
\nabla^\mathcal B_f \begin{pmatrix}q\\ 0\end{pmatrix}_\ga= \begin{pmatrix}\nabla^\gamma_fq -\tau^\C_{-2}(f,q)\\ -\rho^\C_{-1}(f,q)\end{pmatrix}_\ga.
\end{equation}
 Now let $v\in\cO(\cV\cC)$ and $\tilde v\in\mathfrak X(\cP)^H$ be a lift of $v$. Since $K(\tilde v,_-)=0$, we have
 \begin{align*}
 \omega([f^h,\tilde v])&=[\omega(f^h),\omega(\tilde v)]+f^h\cdot\omega(\tilde v)-\tilde v\cdot\omega(f^h)
  \end{align*}
  The first term has values in $W^{-2}$, the second in $\g$ and the third in $W^{-1}$.
  Therefore, we get
 \begin{equation}\label{formula_Bott_in_terms_of_gamma_2}
\nabla^\mathcal B_f \begin{pmatrix}0\\ v\end{pmatrix}_\ga=\begin{pmatrix}\{f,v\}\\ \nabla^\gamma_f v\end{pmatrix}_\ga=\begin{pmatrix}\mathcal L(f,v)\\ \nabla^\gamma_f v\end{pmatrix}_\ga,
\end{equation}
 which finishes the proof.
 \end{proof}

 \begin{cor}\label{cor_cubic_tor_to_curv}
   In the setting of Theorem \ref{Thm_Bott_in_terms_of_gamma}, one has:
\begin{enumerate}
\item There exists a local nowhere vanishing section $f_o$ of $\F$ 
 such that $\nabla^\gamma_{f_o} f_o=0$.
\item For $f_o$ as in (1), we have $$\nabla_{f_o}^\gamma v=\nabla_{f_o}^{f_o}v+\frac{1}{2}\phi_{f_o}^{-1}(\tau^\cC_{-2}(f_o,\phi_{f_o}(v))).$$ 
\item For $f_o$ as in (1), $\tilde\mu^{f_o}(q)(f_o)$ is given by 
$$
 \rho^\cC_{-1}(f_o,q)-\frac{1}{2}\phi_{f_o}^{-1}((\nabla^\gamma_{f_o}\tau_{-2}^\C)(f_o,q))-\frac{1}{4}\phi_{f_o}^{-1}(\tau_{-2}^\C(f_o,\tau_{-2}^\C(f_o,q)))
 $$
 \end{enumerate}
 \end{cor}
 \begin{proof}
(1) holds as finding $f_o$ such that $\nabla^\gamma_{f_o} f_o=0$ amounts to solving
 an ODE. For (2) note that for $v\in\mathcal{O}(\cV\cC)$ we clearly have $v=\binom{0}{v}_\ga\in\cO(\cV\cC)$. 
 So Theorem \ref{Thm_Bott_in_terms_of_gamma} shows that 
\begin{equation}\label{eq_split_ga}
\bar{q}([f_o,v])=\begin{pmatrix}\cL(f_o,v)\\ \nabla^\ga_{f_o}v\end{pmatrix}_\ga,
\end{equation}
where $\bar q: T^{-2}\cC\rightarrow T^{-2}\C/\F$ is the natural projection.  Thus we
can apply Theorem \ref{Thm_Bott_in_terms_of_gamma} again to compute
$\bar{q}([f_o,[f_o,v]])$ in terms of the splitting associated to $\ga$. By
\eqref{formula_Bott_in_terms_of_gamma} the top component $q_{-2}([f_o,[f_o,v]])$
equals
$$
\nabla^\ga_{f_o}\Cal L(f_o,v)-\tau^\cC_{-2}(f_o,\Cal L(f_o,v))+\Cal
L(f_o,\nabla^\ga_{f_o}v). 
$$
From Proposition \ref{prop_part_conn}, we know that  $\mathcal
L=\{_-,_-\}:\F\otimes\cV\C\rightarrow \textrm{gr}_{-2}(T\C)$ is parallel for
$\nabla^\ga$, so the first and last terms in this sum agree and we get
$$
\cL(f_o,\nabla^\ga_{f_o}v)=\tfrac12q_{-2}([f_o,[f_o,v]])+\tfrac12 \tau^\cC_{-2}(f_0,\Cal L(f_o,v)).
$$
On the other hand, from
Proposition \ref{prop_splitting_f}, we know that $\nabla^{f_o}$ is characterized by 
$$\cL(f_o,\nabla_{f_o}^{f_o}v)=\frac{1}{2}q_{-2}([f_o,[f_o,v]]]),$$
 which together with the above implies (2).

To prove (3), we have to compare the splittings induced by $f_o$ and $\ga$. From the formula \eqref{formula_Bott} for the Bott connection in terms of the splitting determined by $f_o$, we see that
$\bar{q}([f_o,v])=\begin{pmatrix}\Cal
L(f_o,v)\\ \nabla^{f_o}_{f_o}v\end{pmatrix}_{f_o}$. Together with \eqref{eq_split_ga} and statement (2), this shows that 
 \begin{equation}\label{change_split}
   \begin{pmatrix}q\\v\end{pmatrix}_{f_o}=
   \begin{pmatrix}q\\ v+\frac{1}{2}\phi_{f_o}^{-1}(\tau^\C_{-2}(f_o,q))\end{pmatrix}_{\gamma}. 
 \end{equation}
 
 By \eqref{change_split} and Theorem
  \ref{Thm_Bott_in_terms_of_gamma}, we therefore get
\begin{align}\label{Bott-comp}\nonumber
\nabla^{\Cal B}_{f_o}\begin{pmatrix}q \\ 0 \end{pmatrix}_{f_o}&=\nabla^{\Cal B}_{f_o}\begin{pmatrix}q \\ \frac{1}{2}\phi_{f_o}^{-1}(\tau^\C_{-2}(f_o,q)) \end{pmatrix}_{\gamma}\\
&=\begin{pmatrix}
\nabla_{f_o}^\gamma
q-\frac{1}{2}\tau_{-2}^\C(f_o,q)\\\frac{1}{2}\nabla_{f_0}^\gamma(\phi_{f_o}^{-1}(\tau^\C_{-2}(f_o,
q))-\rho^{\cC}_{-2}(f_o,q) 
\end{pmatrix}_\ga .
\end{align}
Since $f_o$ and $\mathcal L: \cV\cC\otimes \F^\gamma\rightarrow  \textrm{gr}_{-2}(T\cC)$ are parallel for $\nabla^\ga$, so is $\phi_{f_o}^{-1}$, which allows us
to rewrite the first term in the bottom component of \eqref{Bott-comp} as
$$\frac{1}{2}\phi_{f_o}^{-1}(\nabla^{\ga}_{f_o}\tau^\C_{-2}(f_o, q))=\frac{1}{2}\phi_{f_o}^{-1}((\nabla^\ga_{f_o}\tau^{\cC}_{-2})(f_o,q)+\tau^{\cC}_{-2}(f_o,\nabla^\ga_{f_o}q)).$$ 
On the other hand, by \eqref{formula_Bott},  the bottom
component of $\nabla^{\Cal B}_{f_o}\binom{q}{0}_{f_o}$ in the splitting determined by $f_o$ equals  $-\tilde\mu^{f_o}(q)(f_o)$. Hence
formula \eqref{change_split} shows that we must get  $-\tilde\mu^{f_o}(q)(f_o)$, when we apply
$\frac12\phi_{f_o}^{-1}(\tau^{\cC}_{-2}(f_0,\_))$ to the top line of
\eqref{Bott-comp} and subtract the result from the bottom row, which proves (3). 
\end{proof}

Recall that the curvature of a connection $\gamma$ on the $G$-structure $(\mathcal P\rightarrow M,\theta)$, associated to the $Z$-isotrivial $\C\subset \mathbb PTM$, can be viewed as a $G$-equivariant map from 
$\mathcal P$ to 
$$\Wedge^2 W^*\otimes \g\subset \Wedge^2 W^*\otimes \mathfrak{gl}(W).$$
Also, recall that for any $z\in Z$ we have canonical isomorphisms
\begin{equation}\label{iso_iso}
\hat z^*\otimes \widehat T_zZ/\hat z\cong T_zZ\cong \g/\g_z,
\end{equation}
where $\g_z$ denotes the stabilizer of $\hat z\subset W$ in $\g$ and the second map
is induced by differentiating $\lambda_z: G\rightarrow Z$, given by $\lambda_z(g)=gz$
at the identity element in $G$.  Hence, for an element $z\in Z$ and an element
$\psi\in\Wedge^2 W^*\otimes \g$, we can consider the composition of
$\psi(\hat z,_-)\vert_{\widehat T_{z}Z}:{\widehat T_{z}Z}\rightarrow \g$ with the
projection $\g\rightarrow \g/\g_z$, which, by skew-symmetry of $\psi$, factors to a
map $\widehat{T}_{z}Z/\hat{z}\rightarrow \g/\g_z$. Using the isomorphism
\eqref{iso_iso}, we see that any element $\psi\in\Wedge^2 W^*\otimes \g$ induces a
linear map
\begin{equation}\label{psi_induced_map}
\bar{\psi}_z: \hat z\otimes\hat z\rightarrow \textrm{Hom}(\widehat{T}_{z}Z/\hat{z},\widehat{T}_{z}Z/\hat{z}),
\end{equation}
for any $z\in Z$. For a characteristic conic connection on an isotrivial cone structure that is induced from a connection on the associated $G$-structure, we can now formulate the following condition for its cubic torsion to vanish:
  \begin{thm}\label{thm_cubic_torsion_zero} 
    Suppose $\C\subset  \mathbb P  TM$ is  a $Z$-isotrivial  cone structure  for some
    closed  homogeneous   complex  submanifold  $Z\subset   \mathbb  P  W$   and  let
    $\cP\subset FM$ be the associated $G$-structure.  Let $\gamma\in\Omega^1(\cP,\g)$
    be a connection  on the $G$-structure and $\F:=\mathcal  F^\gamma$ the associated
    conic connection  on $\cC$.  Assume that  the torsion $\tau^M$  of $\gamma$  is a
    section of $\cP\times_G \widetilde\Xi_Z\subset \Wedge^2 T^*M\otimes TM$, where
 \begin{equation}\label{tildeXi}
   \widetilde{\Xi}_Z:=\{\phi\in\Wedge^2 W^*\otimes W: \phi(\hat z, \wh T_{z}Z)
   \subset\hat z\quad\forall z\in Z\},
\end{equation}
so $\F$ is characteristic by Theorem \ref{thm_characteristic_torsion}.  Then the
following statements are equivalent:
\begin{enumerate}
\item The cubic torsion of $\F$ vanishes.
\item For any $f\in\mathcal O(\F)$, the section
  $\rho^{\cC}_{-1}(f,_-)\in\emph{Hom}(\emph{gr}_{-2}(T\C),\cV\cC)$ is pure
  trace with respect to \eqref{Levi_iso}.
\item
The curvature $\rho^M$ of $\gamma$ is a section of $\cP\times_G \Theta_Z\subset\Wedge^2 T^*M\otimes \emph{End}(TM)$,
where 
\begin{equation}\label{Theta_Z}
\Theta_Z:=\{\psi\in\Wedge^2 W^*\otimes \g: \emph{Im}(\bar{\psi}_z)\subset\mathbb C\emph{Id}\quad\forall z\in Z\},
\end{equation}
where $\emph{Im}(\bar{\psi}_z)$ denotes the image of $\bar{\psi}_z$ as defined in \eqref{psi_induced_map}.
 \end{enumerate}
 \end{thm}
 \begin{proof}
   Fix a point $z_0\in Z\subset \mathbb P W$ and let $H<\textrm{Aut}(\wh Z)=:G$ be
   the stabilizer of $z_0$ so that $Z\cong G/H$. The assumption \eqref{tildeXi} is
   equivalent to $\tau^\cC(f,\xi)\in \cO(\F)$ for all $f\in\cO(\F)$ and
   $\xi\in \cO(T^{-2}\C)$, which implies on the one hand, by Theorem
   \ref{thm_characteristic_torsion}, that $\F$ is characteristic and on the other
   hand that all torsion terms on the left hand side of the identity in (3) of
   Corollary \ref{cor_cubic_tor_to_curv} vanish. Hence, (3) of Corollary
   \ref{cor_cubic_tor_to_curv} and (2) of Proposition \ref{tilde_rho} imply the
   equivalence of the three conditions.   
  \end{proof}
  
  In particular, we have:
  
  \begin{cor} Under the assumptions of Theorem \ref{thm_cubic_torsion_zero}, suppose
    that $\rho^M$ is a section of
    $\cP\times_G \widetilde{\Theta}_Z\subset\Wedge^2 T^*M\otimes \emph{End}(TM)$,
    where
\begin{equation}\label{tildetheta_Z}
\widetilde\Theta_Z:=\{\psi\in\Wedge^2 W^*\otimes \g: \psi(\hat z, \wh T_{z}Z)\in\g_z\quad \forall z\in Z\}\subset \Theta_Z.
\end{equation}
Then the cubic torsion of $\F$ vanishes.
  \end{cor}

  \section{An application to cone structures of subadjoint type}
  We will now apply the results from the previous section to study $Z$-isotrivial
  cone structures, where $Z$ is a subadjoint variety. To each simple Lie algebra $\s$
  of type different from $A_n$ and $C_n$ one can associate a homogenous subvariety
  $Z\subset\P W$, called the subadjoint variety of type $\s$. A conceptual
  description of these varieties and their properties needs some structure theory of simple Lie algebras, which we will
  discuss below. In all cases the subadjoint variety turns out to be a certain compact
  Hermitian symmetric space equipped with a specific projective embedding. An explicit description of the subadjoint varieties
  discussed here can for instance be  found in Proposition 2.5 of
  \cite{Hwang-Li-1}. 
  
\subsection{Structure theory of parabolic subalgebras}
Suppose that $\l$ is a complex semisimple Lie algebra and let us fix a Cartan
subalgebra $\h\leq \l$ and a system $\Delta^0$ of simple roots for the set of roots
$\Delta$ determined by $\h$; as usual, we denote by $\l_\beta$ the $1$-dimensional
root space corresponding to $\beta\in\Delta$.

Recall that conjugacy classes of parabolic subalgebras of $\l$ are in
bijection with subsets of simple roots $\Sigma\subset \Delta^0$, which in turn are in
bijection with $|k|$-gradings on $\l$ up to conjugacy by an inner automorphism of
$\l$, see \cite{csbook} for details. In particular, given a subset
$\Sigma\subset \Delta^0$ and a root $\beta\in\Delta$, we denote by
$\textrm{ht}_\Sigma(\beta)$ the sum of all the coefficients of elements of $\Sigma$
in the expression of $\beta$ as a linear combination of simple roots. Then $\Sigma$
gives rise to a $|k|$-graded Lie algebra structure as follows:
\begin{equation}\label{k-grading}
\l=\l_{-k}\oplus \cdots\oplus \l_{-1}\oplus \l_0\oplus \l_1\oplus \cdots \oplus\l_l\quad\quad [\l_i,\l_j]\subset\l_{i+j},
\end{equation}
where $\l_i:=\oplus_{\textrm{ht}_{\Sigma}(\beta)=i}\,\l_{\beta}$ for $i\neq0$ and
$\l_0:=\h\oplus\bigoplus_{\textrm{ht}_{\Sigma}(\beta)=0}\l_{\beta}$.
Moreover,
$\l^0:=\l_0\oplus\l_l\oplus...\oplus\l_k$ is a standard parabolic subalgebra with
respect to $(\h,\Delta^0)$.  It has Levi subalgebra $\l_0$ and nilradical
$\l_+=\l_1\oplus...\oplus\l_k$.  Conversely, any parabolic subalgebra of $\l$ is
conjugate to a parabolic subalgebra that is standard with respect to $(\h,\Delta^0)$,
i.e. that arises from a subset $\Sigma\subset \Delta^0$ as described above.
Evidently, by the grading property, each $\l_{i}$ is an $\l_0$-module, and the
Killing form of $\l$ induces an isomorphism $\l_{-i}^*\cong\l_i$ of
$\l_0$-modules. For a grading as in \eqref{k-grading} we will denote by
\begin{equation}\label{Ind_filt}
\l^{i}:=\bigoplus_{j\geq i}\l_{j} 
\end{equation}
the induced $\l^0$-invariant filtration.

\subsection{Contact gradings and subadjoint varieties}\label{sec_contact_gr}
Suppose that $\s$ is a complex simple Lie algebra. Then there exists a unique, up to
inner automorphisms, \emph{contact grading} on $\s$ \cite{csbook,Yamaguchi}, that is,
a $|2|$-grading
\begin{equation}\label{c-grading}
\s=\s_{-2}\oplus \s_{-1 }\oplus \s_0\oplus \s_1 \oplus\s_{2}
\end{equation}
such that $\dim(\s_{\pm2})=1$ and the Lie bracket
$\s_{\pm1}\times\s_{\pm1}\to\s_{\pm2}$ is non-degenerate. Let
$S\subset \textrm{GL}(\s)$ be the adjoint group of $\s$ and denote by $S^0<S$ the
subgroup preserving the filtration associated to \eqref{c-grading} as defined in
\eqref{Ind_filt}. It is a parabolic subgroup with Lie algebra $\s^0$. Denote by
$S_0<S^0$ the Levi subgroup of $S^0$, which is defined as the subgroup of $S^0$
preserving the grading \eqref{c-grading}.  Consider the natural representation of
$S_0$ on $\s_{-1}$, which is well-known to be faithful. Note that the Lie bracket
$\s_{-1}\times\s_{-1}\to\s_{-2}$ can be viewed a symplectic bilinear form on
$\s_{-1}$ defined up to scale. Since the latter is evidently preserved by $S_0$, the
representation of $S_0$ on $\s_{-1}$ defines an inclusion
\begin{equation}\label{rep_subadjoint}
S_0\hookrightarrow \textrm{CSp}(\s_{-1})\hookrightarrow \textrm{GL}(\s_{-1}),
\end{equation}
where $\textrm{CSp}(\s_{-1})\subset \textrm{GL}(\s_{-1})$ denotes the subgroup that preserves the conformal class of symplectic forms on $\s_{-1}$.

Suppose now that $\s$ is different from type $A_n$ and $C_n$ (which also excludes $B_2$ and $D_3$). Then the subset $\Sigma\subset\Delta^0$ corresponding to the contact grading consists of a
single long simple root, which we denote by $\alpha$.  From the description of the
grading in terms of roots it follows readily that Lie algebra $\s_0$ of $S_0$ is reductive with $1$-dimensional center and that the representation \eqref{rep_subadjoint} 
is irreducible with highest weight $-\alpha$. Simplifying also notation for later, a subadjoint variety is defined as follows.

\begin{defin}\label{def_subadjoint} Let $\s$ be a complex simple Lie algebra of type
  different from $A_n$ and $C_n$ equipped with its contact grading \eqref{c-grading} and set $W:=\s_{-1}$ and $L:=\s_{-2}$. 
  \begin{itemize}
 \item  The
  \textit{subadjoint variety} $Z\subset\P W$ of type $\s$ is the unique closed orbit of the
  action of $S_0$ on $\P W$ induced by \eqref{rep_subadjoint}. So this is the orbit of the point $z_0\in\P W$
  defined by the root space $\s_{-\alpha}\subset\s_{-1}=W$.
  \item We denote by $b\in\Wedge^2 W^*\otimes L$ the conformal class of symplectic forms on $W$ given by the Lie bracket 
  $\s_{-1}\times\s_{-1}\rightarrow\s_{-2}$.
  \end{itemize}
  \end{defin}

  To proceed further, consider the $\s^0$-representation $\s^{-1}/\s^0$, which
  factors to the representation of $\s_0$ on $\s_{-1}$, and denote by $\q\leq \s^0$
  the stabilizer of the line $\hat z_0=\s_{-\alpha}\subset\s_{-1}\cong\s^{-1}/\s^0$,
  which is again a parabolic subalgebra of $\s$. The we have:

\begin{prop}\label{prop_contact_gr} 
  Suppose $\s$ is a complex simple Lie algebra $\neq A_n, C_n$ and fix a pair
  $(\h,\Delta^0)$ of a Cartan subalgebra and system of simple roots. Denote by
  $\alpha\in\Delta^0$ the long simple root giving rise to the contact grading on
  $\s$ and by $\q<\s^0<\s$ the stabilizer in $\s^0$ of
  $\s_{-\alpha}\subset \s_{-1}\cong \s^{-1}/\s^0$.
\begin{enumerate}
\item The standard parabolic subalgebra $\q$ corresponds to the subset
  $\Sigma_{\q}\subset \Delta^0$ consisting of $\alpha$ and all simple roots connected
  to $\alpha$ in the Dynkin diagram of $\s$, see Table
  \ref{tab_contact_roots}. Moreover, $k=5$ for the grading \eqref{k-grading} on $\s$
  induced by $\Sigma_{\q}$.
\item Consider the contact grading $\s=\bigoplus_{i=-2}^{2}\s_i$ determined by
  $\alpha$ and the grading induced by $\Sigma_{\q}$, which we denote by
  $\s=\bigoplus_{j=-5}^{5}\q_{j}$. Then $\q_{\pm1}=\q_{\pm 1}^F\oplus\q_{\pm 1}^V$,
  where $\q_{\pm 1}^F=\s_{\pm\alpha}$ and $\q_{\pm 1}^V=\s_0\cap\q_{\pm 1}$, and 
\begin{equation}\label{s-q-1}
\s_{\pm 1}=\q_{\pm4}\oplus \q_{\pm 3}\oplus\q_{\pm 2}\oplus\q_{\pm 1}^F \quad\quad \s_{\pm2}=\q_{\pm5},
\end{equation}
\begin{equation}\label{s-q-2}
\quad\quad\s_0=\q_{-1}^V\oplus\q_0\oplus\q_1^V.
\end{equation}
Moreover, $$\dim(\q_{\pm 4})=\dim(\q_{\pm 1}^F)=\dim(\q_{\pm 5})=1$$ and $$\dim(\q_{\pm 1}^V)=\dim(\q_{\pm 2})=\dim(\q_{\pm 3}).$$
\item The Lie bracket $[_-,_-];\s_{\pm 1}\times\s_{\pm 1}\rightarrow\s_{\pm 2}$ has the following properties:
\begin{enumerate}
\item it restricts to isomorphisms $\q_{\pm 1}^F\otimes\q_{\pm 1}^V\cong\q_{\pm 2}$
  of $\q_0$-modules;
\item its restrictions to $\q_{\pm 1}^F\times(\q_{\pm 2}\oplus\q_{\pm 3})$ and  $\q_{\pm 2}\times\q_{\pm 2}$ vanish;
\item it induces isomorphisms of $\q_0$-modules
$\q_{\pm 1}^F\cong \q_{\pm 4}^*\otimes \q_{\pm 5}$ and  $\q_{\pm 2}\cong \q_{\pm 3}^*\otimes \q_{\pm 5}$ .
\end{enumerate}
\item There exists a unique element $E\in\q_0$ such that $[E,X]=iX$ for any
  $X\in\q_i$.
\end{enumerate}
\end{prop}

\begin{table}[h!]
\centering
\small
\begin{tabular}{||p{1cm}|p{2.5cm}|p{2.5cm}|p{2cm}|p{2cm}||}
\hline\
 $\s$ & $\alpha$ &$\Sigma_\q$ &$\s_0^s=\g^{s}$ & $\s_{-1}=W$\\
 \hline\hline
 &&&&\\
 $B_n$, $n\geq 3$& \dynkin[parabolic=2]{B}{}  &
 \begin{dynkinDiagram}{B}{xxx*.**}\end{dynkinDiagram} & $A_1\times B_{n-2}$ & $\mathbb C^2\otimes \mathbb C^{4n-6}$ \\
&&&&\\
 $D_n$, $n\geq 5$  &\dynkin[parabolic=2]{D}{} &\begin{dynkinDiagram}{D}{xxx*.***}\end{dynkinDiagram}& $A_1\times D_{n-2}$& $\mathbb C^2\otimes \mathbb C^{4n-8}$ \\
 &&&&\\
 $D_4$ &\dynkin[parabolic=2]{D}{4}&\begin{dynkinDiagram}{D}{xxxx}\end{dynkinDiagram}& $A_1\times A_{1}\times A_1$&$\mathbb C^2\otimes \mathbb C^2\otimes \mathbb C^2$\\
 &&&&\\
$E_6$&   \dynkin[parabolic=2]{E}{6}   &\begin{dynkinDiagram}{E}{*x*x**}
 \end{dynkinDiagram}
&$A_5$ &$\Wedge^3\mathbb C^6$\\
 &&&&\\
 $E_7$&  \dynkin[parabolic=1]{E}{7}    &\begin{dynkinDiagram}{E}{x*x****}
 \end{dynkinDiagram}
 &$D_6$ & $S_+=\mathbb C^{32}$ \\
  &&&&\\
 $E_8$&  \begin{dynkinDiagram}{E}{*******x}
 \end{dynkinDiagram}
 &\begin{dynkinDiagram}{E}{******xx}
 \end{dynkinDiagram}
 &$E_7$&$\mathbb C^{56}$\\
  &&&&\\
 $F_4$&  \begin{dynkinDiagram}{F}{x***}
 \end{dynkinDiagram}
   &\begin{dynkinDiagram}{F}{xx**}
 \end{dynkinDiagram}
 &$C_3$ &$\Wedge^3_0\mathbb C^6$\\
  &&&& \\
 $G_2$&  \begin{dynkinDiagram}{G}{x*}
 \end{dynkinDiagram}
   &\begin{dynkinDiagram}{G}{xx}
 \end{dynkinDiagram}& $A_1$& $S^3\mathbb C^2$\\
 \hline
\end{tabular}
\newline
\caption{Contact roots $\alpha$, $\Sigma_{\q}$, and the $\g^s$-module $W$}
\label{tab_contact_roots}
\end{table}

\begin{proof}
  The first statement of (1) follows from basic results about parabolic subalgebras
  arising as stabilizers of highest weight lines, see e.g. \cite[Prop.3.2.2]{csbook},
  and has also already been established in \cite[Lemma 2.14]{H-N}.  The integer $k$
  in a gradation \eqref{k-grading} always equals $\textrm{ht}_\Sigma(\theta)$, where
  $\theta$ is the highest root of $\s$, and hence the second statement of (1) follows
  by decomposing $\theta$ into simple roots and adding the coefficients of elements
  of $\Sigma_\q$, see also \cite[Prop.2.21(2)]{H-N}. The identities \eqref{s-q-1} and
  \eqref{s-q-2} were proved in (the proof of) Proposition 2.17 and Proposition 2.21
  of \cite{H-N}, where the grading components determined by $\alpha$ are denoted by 
  $\p_i$ instead of $\s_i$.  The statement (3a) was proved in Corollary 2.22 of
  \cite{H-N}, which shows in particular that $\dim(\q_{\pm 2})=\dim(\q_{\pm
    1}^V)$. Moreover, it can be easily checked directly that there is a single root
  $\gamma$ in $\Delta$, namely $\gamma=-\alpha+\theta$, such that
  $ht_{\Sigma_{\q}}(\pm\gamma)=\pm4$, which shows that $\dim(\q_{\pm4})=1$. Since the Lie
  bracket $\s_{\pm 1}\times\s_{\pm 1}\rightarrow\s_{\pm 2}$ is non-degenerate and
  skew-symmetric, $\dim(\s_{\pm 1})$ is even dimensional, which forces
  $\dim(\q_{\pm 2})=\dim(\q_{\pm 3})$ and hence completes the proof of (2). The
  statement (3b) follows from the fact that $[\q_{-i}, \q_{-j}]\subset \q_{-i-j}$ for
  any $i,j$ and that $[\s_{-1},\s_{-1}]\subset \s_{-2}=\q_{-5}$. Statement (3c)
  follows from (3b) and non-degeneracy of
  $\s_{\pm 1}\times\s_{\pm 1}\rightarrow\s_{\pm 2}$. Statement (4) is a general fact
  about $|k|$-gradings, see e.g. \cite[Proposition 3.1.2]{csbook}.
\end{proof}

For a subadjoint variety $Z\subset \mathbb P W$ let us denote by $\g:=\mathfrak{aut}(\widehat Z)$ the Lie algebra of the automorphism group 
$G:=\textrm{Aut}(\widehat{Z})\subset \textrm{GL}(W)$. The last two columns of Table \ref{tab_contact_roots} and Fact 3.1.
of \cite{MS} imply that $$\g=\s_0\subset \mathfrak c\s\p(W)\subset \mathfrak{gl}(W).$$

From Proposition \ref{prop_contact_gr}, we see that the stabilizer $\p$ of
$\hat z_0\subset \s_{-1}$ in $\g=\s_0$ equals $\q\cap \s_0=\q_0\oplus\q_{1}^V$.
This is a parabolic subalgebra of the reductive Lie algebra $\g$ with abelian
unipotent radical $\q_{1}^V$ as $[\q_1^V,\q_{1}^V]\in\q_2\cap \s_0=\{0\}$. Hence,
the subadjoint variety  $$Z\cong G/P\cong G^s/P'$$ is indeed a compact Hermitian symmetric space, where $P$ denotes the stabilizer of $z_0\in\mathbb P W$ in $G$, $G^s$ denotes the
semisimple part of $G$ and $P'=P\cap G^s$.
By part (2) of Proposition \ref{prop_contact_gr}, $W$ inherits a $\q_0$-invariant
grading, which we write as
\begin{equation}\label{gr_W}
W=W_{-4}\oplus W_{-3}\oplus W_{-2}\oplus W_{-1}=\q_{-4}\oplus \q_{-3}\oplus\q_{-2}\oplus\q_{-1}^F.
\end{equation}
The associated $\p$-invariant filtration is denoted by
\begin{equation}\label{filtr_W}
 W=W^{-4}\supset W^{-3}\supset W^{-2}\supset W^{-1}=W_{-1}
\end{equation}
where $W^{-i}=\bigoplus_{j\leq i} W_{-j}$. Using this notation, Proposition \ref{prop_contact_gr}
implies:

\begin{cor}\label{cor_contact_gr} 
The filtration \eqref{filtr_W} has the following properties
\begin{enumerate}
\item
$W^{-1}=\hat z_0=\s_{-\alpha}$ 
\item 
$W^{-2}=\q_{-1}^F\oplus\q_{-2}=[\g,\hat z_0]=\widehat T_{z_0} Z$.
\item 
$\g\cdot W^{-i}\subset W^{-i-1}$
\end{enumerate}
Moreover, $W^{-2}$ is a Lagrange subspace and $W^{-3}$ the annihilator of $W^{-1}$ with respect to the conformal class of symplectic forms $b\in \Wedge^2 W^*\otimes L$, see
Definition \ref{def_subadjoint}. 
\end{cor}

\begin{rem} Suppose $Z\subset \mathbb PW$ is as subadjoint variety as in Definition
  \ref{def_subadjoint}. Fix a point $z_0\in Z$ and let $P<G=\textrm{Aut}(\widehat Z)$
  be the stabilizer of $\hat z_0$.  Then the $\p$-invariant filtration
  \eqref{filtr_W} on $W$, which is the natural filtration on the $\p$-module $W$ as
  defined in (2) of Proposition 3.2.12 of \cite{csbook}, coincides with the
  osculating filtration, of the projective variety $Z\subset \mathbb PW$ determined
  by $z_0$, as defined for instance in Section 2.1 of \cite{LM}. The properties
  listed in Corollary \ref{cor_contact_gr} are therefore well-known in slightly
  different terms, see for instance Proposition 2.6 of \cite{Hwang-Li-1}.
\end{rem}

\subsection{Cone structures of subadjoint type}
We will now consider the following cone structures:

\begin{defin}\label{def_cone_of subadjoint_type} 
  A cone structure $\C\subset \mathbb P TM$ on a complex manifold $M$ is called of
  \emph{subadjoint type}, if it is $Z$-isotrivial for $Z\subset \mathbb PW$ a
  subadjoint variety as in Definition \ref{def_subadjoint}.
\end{defin}

The associated $G$-structure to cone structure of subadjoint type, with
$G=\textrm{Aut}(\widehat{Z})$, is a \emph{complex parabolic almost conformally
  symplectic structure}, abbreviated complex PACS-structure, see Definition 1.3 of
\cite{Cap_Salac}.  Conversely, consider a complex PACS-structure
$\mathcal P\subset FM$ of type different from $A_n$ ($C_n$ is already excluded in
\cite{Cap_Salac}) with structure group $G=\textrm{Aut}(\widehat{Z})$. Then forming
the associated bundles
$\C:=\cP\times_{G}Z\subset \cP\times_{G}\mathbb P W=\mathbb P TM$ defines a cone
structure of subadjoint type. Hence, there is a one-to-one correspondence between
cone structures of subadjoint type and complex PACS-structures of type $\neq A_n$.

In \cite{Cap_Salac} it was shown that PACS-structures admit canonical connections. To
recall that more precisely, we denote, for a contact grading \eqref{c-grading}, by
$\Wedge^3_0\s_{-1}^*\subset \Wedge^3\s_{-1}^*$ the kernel of the surjective
$\s_0$-equivariant linear map
$\Wedge^3\s_{-1}^*\rightarrow \s_{-1}^*\otimes\s_{-2}^*$ given by contraction with
$b^{-1}\in\Wedge^2\s_{-1}\otimes\s_{-2}^*$.  Note that
$ \Wedge^3_0\s_{-1}^*\otimes\s_{-2}\subset \Wedge^3\s_{-1}^*\otimes\s_{-2}$ can be
naturally viewed as a subspace of $\Wedge^2\s_{-1}^*\otimes\s_{-1}$ via the
isomorphism $\s_{-1}\cong \s_{-1}^*\otimes\s_{-2}$ of $\s_0$-modules induced by the
Lie bracket.

\begin{thm}\label{thm_can_conn}\cite[Thm.\,4.2]{Cap_Salac}
Suppose that $\s$ is a complex simple Lie algebra $\neq  C_n$ equipped with its contact grading. 
Then the 
Spencer map $\partial_S: \s_{-1}^*\otimes\s_0\rightarrow \Wedge^2\s_{-1}^*\otimes\s_{-1}$ corresponding to the representation $\s_0\hookrightarrow \mathfrak{gl}(\s_{-1})$ is injective and as $\s_0$-module
\begin{equation*}
\Wedge^2\s_{-1}^*\otimes\s_{-1}=\emph{im}(\partial_S)\oplus \ker(\square)\oplus (\Wedge^3_0\s_{-1}^*\otimes\s_{-2}).
\end{equation*}
Here $\ker(\square)\cong H^2(\s_-,\s)$ is the kernel of the Kostant Laplacian in
degree $2$ associated to the standard complex
$(\Wedge^*\s_{-}^*\otimes \s,\partial _K)$ computing the Lie algebra cohomology of
the nilpotent Lie algebra $\s_-:=\s_{-2}\oplus\s_{-1}$ with coefficients in $\s$, see
\cite[Section 4]{Cap_Salac} for more details.
\end{thm}

Using the notation from Section \ref{sec_contact_gr}, an immediate consequence of Theorem \ref{thm_can_conn} is:
\begin{cor}\cite[Cor.4.3]{Cap_Salac}\label{cor_can_connection}
Any PACS-structure admits an unique connection whose torsion is a section of
$\cP\times_G T$, where
$$T:=\ker(\square)\oplus (\Wedge^3_0 W^*\otimes L)\subset \Wedge^2 W^*\otimes W.$$ 
We refer to it as the canonical connection of the PACS-structure on $M$.
\end{cor}

\begin{rem}\label{Torsion-G_2}
For PACS-structures of type $\s=G_2$, one has
$\Wedge^3_0 W^*=\{0\}$ for dimensional reasons and the torsion module $T$ simply equals $\ker(\square)$.
\end{rem}

We now want to determine the characteristic torsion of the conic connection induced
by the canonical connection of the PACS-structure on the associated cone structure of
subadjoint type. We start by making some aspect of the proof of Theorem
\ref{thm_can_conn} in \cite[Thm.\,4.2]{Cap_Salac} more explicit.

\begin{lem}\label{lem_Xi_Spencer}
Suppose $Z\subset \mathbb P W$ is a subadjoint variety as in Definition \ref{def_subadjoint}. Denote by $\g:=\mathfrak{aut}(\widehat Z)\subset\mathfrak{gl}(W)$ the Lie algebra of $G:=\emph{Aut}(\widehat Z)$ and by $\partial_S:
W^*\otimes\g\rightarrow \Wedge^2 W^*\otimes W$ the Spencer map of the representation
$\g\hookrightarrow \mathfrak{gl}(W)$.  For a $G$-invariant subspace $V\subset \Wedge^2 W^*\otimes W$ the following are equivalent:
\begin{enumerate}
\item $V\subset \im(\partial _S)$
\item
$V\cap  \ker(\square)=\{0\}$  and $\emph{alt}(V)\cap \Wedge^3_0W^*\otimes L=\{0\}.$
\end{enumerate}
Here $ \emph{alt}: \Wedge^2 W^*\otimes W\rightarrow \Wedge^3W^*\otimes L$ is the composition of the isomorphism $\Wedge^2 W^*\otimes W\cong \Wedge^2 W^*\otimes W^*\otimes L$, induced by $b$, with the alternation map.
\end{lem}
\begin{proof}
 The proof of Theorem \ref{thm_can_conn} in \cite[Thm.\,4.2]{Cap_Salac} shows
 immediately that (1) implies (2). Conversely, assume that (2) holds for a
 $G$-invariant subspace $V\subset \Wedge^2 W^*\otimes W$ and consider the diagram
 from the proof of \cite[Thm.\,4.2]{Cap_Salac}. This decomposes the Lie algebra
 cohomology differential $\partial_K$, mentioned in Theorem \ref{thm_can_conn}, whose cohomology in degree two is
 $\ker(\square)$ and in our notation, it reads as:
 $$
 \begin{tikzcd}
   W^* \arrow{r}{}\arrow{dr}{j} & W^*\otimes \mathfrak g\arrow{r}{\partial_S}\arrow{dr} &\Wedge^2W^*\otimes W\arrow{r}{\textrm{alt}} &\Wedge^3W^*\otimes L \\
   & L^*\otimes W \arrow{r}{} \arrow{ur} &W^*\otimes L^*\otimes L\arrow{ur}{i} &
 \end{tikzcd}
$$ Up to a non-zero factor, $i$ coincides with the inclusion of the trace part, so
for $\phi\in V$ the assumption that $\textrm{alt}(V)\cap \Wedge^3_0W^*\otimes
L=\{0\}$ implies $\textrm{alt}(\phi)=i(\tilde\phi)$ for some $\tilde\phi\in
W^*\otimes L^*\otimes
L$, so
$$\phi-\tilde\phi\in\ker(\partial_K)=\ker(\square)\oplus\im(\partial_K).$$
Hence, we can write $\phi-\tilde\phi=\partial_K(\psi_1)+\psi_2$ for some
$\psi_2\in\ker(\square)$ and, since $j$ is an isomorphism and $\partial_K^2=0$,
we can moreover assume $\psi_1$ to be an element of $W^*\otimes\g$. The diagram now
implies that $\phi=\partial_S(\psi_1)+\psi_2$. Hence, $V\subset
\im(\partial_S)\oplus\ker(\square)\subset\Wedge^2W^*\otimes W$. Now Kostant's Theorem
\cite{Kostant},\cite[Theorem 3.3.5]{csbook} shows in particular that all irreducible
components of $\ker(\square)$ appear with multiplicity $1$ in $\Wedge^2W^*\otimes
W$. Hence, $G$-invariance of $V$ and $V\cap \ker(\square)=\{0\}$ imply that
$V\subset\im(\partial_S)$.
\end{proof}

\begin{thm}\label{thm_Xi_Spencer} 
  In the setting of Lemma \ref{lem_Xi_Spencer} suppose that $Z$ is not of type
  $B_3$. Then the $G$-invariant subspace $\Xi_Z\subset \Wedge^2 W^*\otimes W$ as
  defined in \eqref{Xi} is contained in the image of $\partial_S$.

\end{thm}
\begin{proof}
For a subadjoint variety $Z\subset \P W$ of type $\s\neq B_3$ we verify that $\Xi_Z$ satisfies condition (2) of Lemma \ref{lem_Xi_Spencer}.  
To do so, let us fix a Cartan subalgebra and subsystem of simple roots $(\h,\Delta^0)$ of $\s$ and
  consider the contact grading determined by $\alpha\in\Delta^0$ as in Proposition
  \ref{prop_contact_gr}, where $W$ and $\g$ are realized as $\s_{-1}$ and $\s_0$. As
  in Section \ref{sec_contact_gr} denote by $P<G$ the stabilizer of the point
  $z_0\in\mathbb PW$ determined by $\s_{-\alpha}$ so that
  $Z=G z_0\cong G/P\subset \mathbb P W$ and denote the induced grading and filtration
  on $W$ as in \eqref{gr_W} and \eqref{filtr_W}.

\textbf{Claim 1}: $\Xi_Z\cap \ker(\square)=\{0\}$.\newline
Kostant's Theorem \cite{Kostant},
\cite[Thm.\,3.3.5]{csbook} gives a description of the $\s_0$-submodule
$$\ker(\square)\cong H^2(\s_-,\s)\subset \Wedge^2 W^*\otimes W$$ in terms of lowest
weights and lowest weight vectors, and implies that $\ker(\square)$ is also
isomorphic to the second homology $H_2(\s_+,\s)$ of the Lie algebra
$\s_+:=\s_1\oplus\s_2$ with coefficients in $\s$. 
More explicitly, it shows that the lowest weight vectors of the
 irreducible components 
 of $\ker(\square)$ are given by
 \begin{equation}\label{Kostant_lowest_weight_vec}
 \s_{\alpha}\wedge\s_{\alpha+\beta}\otimes\s_{-\theta+\alpha}\subset \q_{1}^F\wedge\q_{2}\otimes\q_{-4}\cong W_{-1}^*\wedge W_{-2}^*\otimes W_{-4},
 \end{equation}
 where $\beta$ is any root connected to $\alpha$ in the Dynkin diagram of $\s$ and
 $\theta$ denotes the highest root of $\s$, see also Proposition 2.10 (2) of
 \cite{H-N}. By Corollary \ref{cor_contact_gr}, elements $\phi\in \Xi_Z$ have in
 particular the property that $\phi(W^{-1}, W^{-2})$ is contained in $W^{-2}$. Hence
 \eqref{Kostant_lowest_weight_vec} shows that none of the lowest weight vectors of
 the irreducible components of $\ker(\square)$ can be contained in $\Xi_Z$, which
 readily implies Claim 1.  

\textbf{Claim 2}: $\textrm{alt}(\Xi_Z)\cap \Wedge^3_0W^*\otimes L=\{0\}$.\newline
It is sufficient to show for $\phi\in\Xi_Z$ the alternation $\textrm{alt}(\phi)$ annihilates any highest weight vector of the irreducible components of 
the $\g$-representation $\Wedge^3_0W$. For $Z$ of type $\s=G_2$ there is nothing to show as $\Wedge^3_0W=\{0\}$, see Remark \ref{Torsion-G_2}. In the other cases, 
the $\g$-representation $\Wedge^3_0W$ is irreducible for $Z$ of type $\s=E_7, E_8, F_4$, has two irreducible components for $\s=E_6, B_{n\geq4}, D_{n\geq 6}$ and three for $\s=D_4,D_5$.
In the following table, we list the form of the highest weight vectors of the
irreducible components of the $\g$-module $\Wedge^3_0W$ in terms of root spaces of
$\s$, where we use the Bourbaki ordering for the simple roots $\alpha_i$.

\begin{table}[h!]
\centering
\begin{tabular}{||c |c | c c||} 
\hline
 $\s$ & $\g^s$ && $\Wedge^3_0 W$ \\ [0.5ex] 
 \hline\hline
 $E_6$& $A_5$ && $\s_{-\alpha_2}\wedge \s_{-\alpha_2-\alpha_4}\wedge \s_{-\alpha_2-\alpha_4-\alpha_3}$\\ 
 
 &  &&  $\s_{-\alpha_2}\wedge \s_{-\alpha_2-\alpha_4}\wedge \s_{-\alpha_2-\alpha_4-\alpha_5}$ \\
 \hline
 $E_7$ & $D_6$ && $\s_{-\alpha_1}\wedge\s_{-\alpha_1-\alpha_3}\wedge\s_{-\alpha_1-\alpha_3-\alpha_4}$  \\
 \hline
 $E_8$& $E_7$ & & $\s_{-\alpha_8}\wedge\s_{-\alpha_8-\alpha_7}\wedge\s_{-\alpha_8-\alpha_7-\alpha_6}$   \\
 \hline
 $F_4$ & $C_3$ & &$\s_{-\alpha_1}\wedge\s_{-\alpha_1-\alpha_2}\wedge\s_{-\alpha_1-\alpha_2-\alpha_3}$  \\ 
 \hline
 $B_{n\geq 4}/D_{n\geq 6}$& $A_1\times B_{n-2}/D_{n-2}$&&$\s_{-\alpha_2}\wedge \s_{-\alpha_2-\alpha_3}\wedge \s_{-\alpha_2-\alpha_3-\alpha_4}$ \\
 &&&$\s_{-\alpha_2}\wedge \s_{-\alpha_1-\alpha_2}\wedge \s_{-\alpha_2-\alpha_3}$ \\
 \hline
 $D_5$&$A_1\times A_3$&& $\s_{-\alpha_2}\wedge \s_{-\alpha_2-\alpha_3}\wedge \s_{-\alpha_2-\alpha_3-\alpha_4}$ \\
 &&&$\s_{-\alpha_2}\wedge \s_{-\alpha_2-\alpha_3}\wedge \s_{-\alpha_2-\alpha_3-\alpha_5}$\\ 
&&& $\s_{-\alpha_2}\wedge \s_{-\alpha_1-\alpha_2}\wedge \s_{-\alpha_2-\alpha_3}$\\
\hline
$D_4$& $A_1\times A_1\times A_1$&&$\s_{-\alpha_2}\wedge \s_{-\alpha_1-\alpha_2}\wedge \s_{-\alpha_2-\alpha_3}$\\
&&&$\s_{-\alpha_2}\wedge \s_{-\alpha_1-\alpha_2}\wedge \s_{-\alpha_2-\alpha_4}$\\
&&&$\s_{-\alpha_2}\wedge \s_{-\alpha_2-\alpha_3}\wedge \s_{-\alpha_2-\alpha_4}$\\
 \hline
  \end{tabular}
  \newline
\caption{Highest weight vectors of $\Wedge^3_0W$}\label{table_Wedge3}
 \end{table}
 Hence, we see from Table \ref{tab_contact_roots} and Proposition
 \ref{prop_contact_gr}, that all highest weight vectors are are elements of
 $W_{-1}\wedge W_{-2}\wedge W_{-2}\subset \Wedge ^3_0W$.  Consider now $\phi\in\Xi_Z$
 and let $w_{-1}\wedge w_{-2}\wedge w'_{-2}\in W_{-1}\wedge W_{-2}\wedge W_{-2}$ be a
 highest weight vector of an irreducible component of $ \Wedge ^3_0W$.  Since
 $\phi(W^{-1}, W^{-2})\subset W^{-2}$ and $W^{-2}\subset W$ is a Lagrange subspace
 with respect to $b$ by Corollary \ref{cor_contact_gr},
 $\textrm{alt}(\phi)(w_{-1}, w_{-2},w'_{-2})=0$ if and only if
 $b(\phi(w_{-2}, w'_{-2}),w_{-1})=0$. By Corollary \ref{cor_contact_gr}, the
 annihilator of $W^{-1}$ with respect to $b$ equals $W^{-3}$.  Hence, to prove Claim
 2 it remains to show that for any highest weight vector
 $w_{-1}\wedge w_{-2}\wedge w_{-2}'\in W_{-1}\wedge W_{-2}\wedge W_{-2}$ of an
 irreducible component of $\Wedge^3_0W$ one has $\phi(w_{-2},w_{-2}')\subset W^{-3}$
 for all $\phi\in\Xi_Z$.
In view of Table \ref{table_Wedge3}, there are only two cases to consider.
 \\
 \textbf{Case 1}: The irreducible component is generated by a highest weight vector
 $$w_{-1}\wedge w_{-2}\wedge
 w'_{-2}\in\s_{-\alpha}\wedge\s_{-\alpha-\beta}\wedge\s_{-\alpha-\beta-\gamma},$$
 where $\alpha$ is the contact root, and in the Dynkin diagram, $\beta$ is connected
 to $\alpha$ and $\gamma$ is connected to $\beta$, but not to $\alpha$. Since in all
 cases any simple root connected to $\alpha$ is connected by a single line in the
 Dynkin diagram, see Table \ref{tab_contact_roots}, $\alpha+2\beta$ can not be a root
 of $\s$. This implies that there exists $e_{-\beta}\in \s_{-\beta}\subset\s_0$ such
 that
\begin{equation*}
\exp(e_{-\beta})(w_{-1})=w_{-1}+[e_{-\beta}, w_{-1}]=w_{-1}+w_{-2}\in \widehat Z.
\end{equation*}
Write $z_1\in Z$ for the point determined by $\exp(e_{-\beta})(w_{-1})$.
Since  $\hat z_1=\mathbb C(w_{-1}+w_{-2})\subset W^{-2}$, 
we conclude that the affine tangent space $\widehat T_{z_1}Z=[\s_0, \hat z_1]$ at $z_1$ is contained in $[\s_0, W^{-2}]=W^{-3}$ by Corollary \ref{cor_contact_gr}. Moreover, since $\gamma$ is not connected to $\alpha$, there exists 
$e_{-\gamma}\in\s_{-\gamma}\subset \s_0$ such that
$$[e_{-\gamma}, w_{-1}+w_{-2}]=[e_{-\gamma},w_{-2}]=w_{-2}'\in \widehat T_{z_1}Z.$$ Hence, $$\phi(\hat z_1, \widehat T_{z_1}Z)\subset \widehat T_{z_1}Z\subset W^{-3}$$ implies in particular that
$\phi(w_{-1}+w_{-2}, w_{-2}')$ is contained in $W^{-3}$, which together with
$\phi(W^{-1}, W^{-2})\subset W^{-2}$ shows that $\phi(w_{-2}, w_{-2}')\in W^{-3}$.
\\
\textbf{Case 2}: The irreducible component is generated by a highest weight
vector
$$w_{-1}\wedge w_{-2}\wedge w_{-2}'
\in\s_{-\alpha}\wedge\s_{-\alpha-\beta}\wedge\s_{-\alpha-\beta'},$$ where $\beta$ and
$\beta'$ are simple roots connected to $\alpha$, but not to each other, in the Dynkin
diagram. As in Case 1, since $\alpha+2\beta$ and $\alpha+2\beta'$ are not roots, we
conclude that there exists $e_{-\beta}\in\s_{-\beta}$ and
$e_{-\beta'}\in\s_{-\beta'}$ such that
\begin{align*}
\exp(\pm e_{-\beta} )(w_{-1})&=w_{-1}\pm [e_{-\beta}, w_{-1}]=w_{-1}\pm w_{-2} \\
\exp(\pm e_{-\beta'} )(w_{-1})&=w_{-1}\pm [e_{-\beta'}, w_{-1}]=w_{-1}\pm w_{-2}' 
\end{align*}
are elements of $\widehat Z$. We denote the corresponding points in $Z$ by $z_{\pm}$ respectively $z'_{\pm}$. Moreover, again as in Case 1, since $\hat z_{\pm }$ and $\hat z'_{\pm}$ are 
in $W^{-2}$, we conclude that their affine tangent spaces $\widehat T_{z_{\pm}} Z$ and  $\widehat T_{z'_{\pm}} Z$ are contained in $W^{-3}$.  Since $\beta+\beta'$ is not a root, this shows in particular 
that 
\begin{align*}
[e_{-\beta'}, w_{-1}\pm w_{-2}]&=w_{2}'\pm w\in  \widehat T_{z_{\pm}} Z\subset W^{-3}\\ 
[e_{-\beta}, w_{-1}\pm w_{-2}']&=w_{2}\pm w\in \widehat T_{z'_{\pm}} Z\subset W^{-3} 
\end{align*}
for some $w\in\s_{-\alpha-\beta-\beta'}$. For $\phi\in\Xi_Z$ we therefore must have
\begin{align}\label{eq1}
\phi(w_{-1}\pm w_{-2}, w_{-2}'\pm w)= &\phi(w_{-1},w'_{-2})+\phi(w_{-2}, w)\\\nonumber
&\pm (\phi(w_{-1}, w)+\phi(w_{-2}, w_{-2}'))\in W^{-3} 
\end{align}
and
\begin{align}\label{eq2}
\phi(w_{-1}\pm w_{-2}', w_{-2}\pm w)=&\phi(w_{-1}, w_{-2})+\phi(w_{-2}', w)\\\nonumber
&\pm (\phi(w_{-1}, w)+\phi(w_{-2}', w_{-2}))\in W^{-3}
\end{align}
Equation \eqref{eq1} implies that $\phi(w_{-1}, w)+\phi(w_{-2}, w_{-2}')\in W^{-3}$ and equation \eqref{eq2} 
that $\phi(w_{-1}, w)+\phi(w_{-2}', w_{-2})=\phi(w_{-1}, w)-\phi(w_{-2}, w_{-2}')\in W^{-3}$. Hence, $\phi(w_{-2}, w_{-2}')\in W^{-3}$.
\end{proof}

\begin{rem}
  In Theorem \ref{thm_Xi_Spencer} we have excluded the case $\s=B_3$. This case is
  different from the cases $B_{n\geq 4}$, since $\alpha_2+2\alpha_3$ is root. Indeed,
  for $B_3$, $\Wedge^3_0W$ contains an irreducible component generated by a highest
  weight vector in
  $\s_{-\alpha_{2}}\wedge\s_{-\alpha_{2}-\alpha_3}\wedge\s_{-\alpha_{2}-2\alpha_3}\in
  W_{-1}\wedge W_{-2}\wedge W_{-3}$ .
\end{rem}

Before we analyze the consequences of Theorem \ref{thm_Xi_Spencer} let us establish
another fact about the projective geometry of a subadjoint variety.

\begin{lem}\label{prop1_projgeom_subadjoint}
Suppose $Z\subset \mathbb P W$ is a subadjoint variety as in Definition \ref{def_subadjoint}. Denote 
by $\mathcal O(1)=\mathcal O(1)\vert_Z$ the restriction to $Z$ of the hyperplane line bundle of projective space $\mathbb P W$. Then the following folds: 
\begin{enumerate}
\item $H^0(Z,\mathcal O (1))=W^*$; 
\item the homomorphism
$\mathfrak{aut}(\widehat Z)\otimes W^*\rightarrow H^0(Z, TZ\otimes\mathcal O(1))$
induced by the restriction homomorphism $\mathfrak{aut}(\widehat Z)\rightarrow H^0(Z,TZ)$ is surjective.
\end{enumerate}
\end{lem}
\begin{proof}
  Let us denote by $G^{s}$ the semisimple part of the reductive Lie group
  $G=\textrm{Aut}(\widehat Z)$  and by $\g^{s}$ and $\g=\mathfrak{aut}(\widehat
  Z)\subset\mathfrak{gl}(W)$ the corresponding Lie algebras.  The first statement
  follows directly from the Bott--Borel--Weil Theorem applied to the complex flag
  variety $Z\cong G/P\cong G^{s}/P'$, where $P$ and $P'$ are the stabilizers in
  $G$ respectively $G^s$ of a point in $Z\subset \mathbb P
  W$, see e.g. Section 3.2.\,of \cite{MS}.  For the second statement, recall that the
  natural homomorphism $\g^s\hookrightarrow \mathfrak{aut}(\widehat Z)\rightarrow
  H^0(Z,TZ)$ identifies $\g^s$ with
  $H^0(Z,TZ)$, see e.g. Fact 3.1.\,of \cite{MS}. By the Bott--Borel-Weil theorem,
  the number of irreducible components of the $G^s$-representation $H^0(Z,
  TZ\otimes\mathcal O(1))$ equals the number of simple factors of
  $\g^s$ and the irreducible components are given by the Cartan products of the
  simple factors of $\g^s$ with
  $W^*$, i.e. their highest weights equal the sums of a highest root of a simple
  factor and of the highest weight of $W^*$ . Moreover, the
  $G^s$-equivariant map in question is the natural projection from $\g\otimes
  W^*\supset \g^s\otimes
  W^*$ to the sum of the Cartan products of the simple factors of $\g^s$ with $W^*$.
\end{proof}

Now Theorem \ref{thm_Xi_Spencer} and Lemma \ref{prop1_projgeom_subadjoint} imply:

\begin{thm}\label{thm_conic_conn_subadjoint} 
  Suppose $\C\subset \mathbb P TM$ is a cone structure of subadjoint type on a
  complex manifold $M$. Then the following holds:
\begin{enumerate}
\item Any conic connection on $\C$ is locally induced by a connection on the $G$-structure associated to $\C$.
\item Assume $\C$ is not of type $\s=B_3$. Then, if $\C$ admits a characteristic
  conic connection $\F$, then the conic connection $\F^{\textrm{can}}$ that is
  induced by the canonical connection in Corollary \ref{cor_can_connection} of the
  associated $G$-structure is characteristic and $\F=\F^{\textrm{can}}$. This happens
  if and only if the associated $G$-structure admits a torsion-free connection.
\end{enumerate}
\end{thm}
\begin{proof}
(1) This follows from Theorem 3.6 of \cite{Hwang-Li-2} and Lemma \ref{prop1_projgeom_subadjoint}.

(2) If $\F$ is a characteristic conic connection, then by (1) it is locally induced
by a principal connection $\gamma$ on the associated $G$-structure $\cP\subset
FM$. Denote by $\gamma^{\textrm{can}}$ the canonical principal connection on $\cP$
from Corollary \ref{cor_can_connection}. Then $\gamma=\gamma^{\textrm{can}}+\psi$,
where $\psi$ is a locally defined section of the bundle $\cP\times_G W^*\otimes \g$.
It follows that the difference between the torsions of $\gamma$ and
$\gamma^{\textrm{can}}$ is given by
$\tau^M_\gamma=\tau^M_{\gamma^{\textrm{can}}}+\partial_S(\psi)$. Then Theorem
\ref{thm_Xi_Spencer} and Theorem \ref{thm_characteristic_torsion} imply that
$\gamma^{\textrm{can}}$ is torsion-free and $\F^{\textrm{can}}$ is
characteristic. Due to Proposition \ref{uniqueness_of_charc_conn},
$\F=\F^{\textrm{can}}$ by uniqueness.
\end{proof}

\begin{thm} Suppose $\C\subset \mathbb P TM$ is a cone structure of subadjoint type
  $\neq B_3$ that admits a characteristic conic connection $\F$. Then the cubic
  torsion of $\F$ vanishes.  The associated $G$-structure is not necessarily flat,
  however.
\end{thm}
\begin{proof}
By Theorem \ref{thm_conic_conn_subadjoint}, $\F$ is induced by the canonical connection $\gamma$ on the associated $G$-structure and the latter must be torsion-free. 
Note that, if $\cC$ is of type $G_2$, the cubic torsion of $\F$ vanishes automatically for dimensional reasons. Hence, we may assume $\cC$ is not of type $G_2$. Then,
by Theorem \ref{thm_cubic_torsion_zero}, 
in order to show that the cubic torsion of $\F$ vanishes, it is sufficient to show that the curvature $\rho$ of the associated $G$-structure, viewed as a $G$-equivariant function $\cP\rightarrow \Wedge^2 W^*\otimes \g$ 
has values in the space $\Theta_Z$ defined in \eqref{Theta_Z}. A description of the curvature $\rho$ for the torsion-free $G$-structures we consider here can be found in \cite{Cahen-Schwachhoefer, Cap_Salac}. Let us recall it here:
consider the semisimple part $\g^{s}$  of $\g$, which equals $\g\cap \mathfrak{sp}(W)$. Then there exists a natural $G$-equivariant inclusion
$$\g^{s}\otimes L^*\hookrightarrow \Wedge^2 W^*\otimes\g^{s};$$
a description in terms of the contact grading on
$\s$ can be found in Section 4.7 of \cite{Cap_Salac}. The curvature of
$\rho$ is then known to have values in the $G$-invariant subspace $K(\g^{s})\subset
\Wedge^2
W^*\otimes\g^{s}$ given by the image of the above inclusion. Let us fix a point
$z_0\in Z\subset \mathbb PW$ and let $P<G=\textrm{Aut}(\wh Z)$ be its
stabilizer.
By (4) of Proposition \ref{prop_contact_gr}, there exists an element $E$ in
$\p$ that acts diagonalizable on $W$ with eigenvalues $-1$, $-2$, $-3$, $-4$, on
$\g$ with eigenvalues $-1$, $0$, $1$ and on $L$ by multiplication by
$-5$, reflecting the decompositions \eqref{s-q-1} and \eqref{s-q-2}. Hence, its
eigenvalues on $K(\g^{ss})$ are $4,5$ and
$6$. If we write
$\g=\g_{-1}\oplus\g_0\oplus\g_1$ for the eigenspace decomposition on
$\g$ so that $\p=\g_0\oplus\g_1$, then $E$ acts by multiplication by
$2$ on $W_{-1}^*\wedge W_{-2}^*\otimes
\g_{-1}$. Therefore, we conclude that an element $\psi\in
K(\g^{s})$ has the property that $\psi(W_{-1},
W_{-2})$ is an element in the stabilizer $\p$ of $\hat
z_0=W^{-1}$. As this is true for any choice $z_0\in Z$, we conclude that
$\rho$ has values in $\widetilde\Theta_Z\subset
\Theta_Z$ as defined in \eqref{tildetheta_Z}, which proves the first claim. By the
classification of irreducible holonomies of torsion-free affine connections in
\cite{MS}, there exist torsion-free non-flat
$G$-structures of subadjoint type,--they are the structures realizing the special
complex symplectic holonomies in the table in Theorem C of \cite{MS}. Each of them
induces a cone structure of subadjoint type with characteristic conic connection with
vanishing cubic torsion, which proves the last claim.
\end{proof}

This is result should be compared to Theorem 1.5 of \cite{Hwang-Li-1}, where it is
shown that any VMRT-structure of subadjoint type $\neq
G_2$ is locally flat. Hence, for a cone structure of subadjoint type $\neq
G_2$ being VMRT is stronger than the existence of a characteristic conic connection
with vanishing cubic torsion.

\begin{rem}
  While the proofs of existence of linear connections with exotic symplectic holonomy
  provide a uniform argument for the existence of non-flat cone structures of
  subadjoint type that admit a characteristic conic connection with vanishing cubic
  torsion, simpler arguments are available for individual examples of structures. The
  point here is that in the setting of exotic holonomy it is crucial to exclude
  symmetric spaces, which in our setting are fine as examples. The simplest instance
  of an example of that type is a complex analog of a structure on the real
  Grassmannian of $2$-planes which is discussed in Theorem D of
  \cite{Cahen-Schwachhoefer}.

  Consider the group $G:=SO(n+2,\Bbb C)$ and let $H$ be the stabilizer of a complex
  two-plane on which the complex bilinear form defining the orthogonal group is
  non-degenerate. Thus $H\cong S(O(2,\Bbb C)\times O(n,\Bbb C))$ and in the standard
  presentation of the Lie algebra $\mathfrak g$ of $G$ as skew symmetric matrices,
  $\mathfrak h$ corresponds to matrices which are block-diagonal with blocks of sizes
  $2$ and $n$. This readily implies that $(\mathfrak g,\mathfrak h)$ is a symmetric
  pair, so there is an $H$-invariant complement $\mathfrak m\subset\mathfrak g$ to
  $\mathfrak h$ and $G/H$ is a complex symmetric space. Now $\mathfrak m$ is
  isomorphic to the space $M_{n,2}(\Bbb C)$ of complex $n\times 2$-matrices and
  $\{0\}\neq [\mathfrak m,\mathfrak m]\subset\mathfrak h$, so the canonical
  connection on $G/H$ is torsion-free but non-flat. From the construction it is also
  clear that the isotropy representation of $H$ on $\mathfrak m$ comes from the
  natural action of $O(2,\Bbb C)\times O(n,\Bbb C)$ on $M_{n,2}(\Bbb C)$ by matrix
  multiplication from both sides. But the corresponding action of
  $S(GL(2,\Bbb C)\times O(n,\Bbb C))$ gives rise to the subadjoint variety for the
  $B_n/D_n$-cases, so $G/H$ carries a canonical cone structure as claimed.
\end{rem}

\end{document}